\documentclass[a4paper,fleqn]{cas-sc}

\usepackage[T1]{fontenc}
\usepackage[utf8]{inputenc}
\usepackage{microtype}
\usepackage{amsmath,amssymb,amsthm,mathtools}
\usepackage{aliascnt}
\usepackage{enumitem}
\usepackage[numbers,sort&compress]{natbib}
\usepackage[nameinlink,noabbrev]{cleveref}
\usepackage{booktabs}

\hypersetup{
  colorlinks=true,
  linkcolor=blue,
  citecolor=blue,
  urlcolor=blue,
  pdftitle={Frobenius-orbit slicing and uniform elimination of positive-dimensional singular loci},
  pdfauthor={Author Name},
  pdfsubject={Uniform higher-jet and complete-intersection Bertini estimates},
  pdfkeywords={Bertini theorem, finite fields, Frobenius orbits, higher jets, singular loci, complete intersections}
}

\allowdisplaybreaks[3]
\numberwithin{equation}{section}

\theoremstyle{plain}
\newtheorem{theorem}{Theorem}[section]
\newaliascnt{proposition}{theorem}
\newtheorem{proposition}[proposition]{Proposition}
\aliascntresetthe{proposition}
\newaliascnt{lemma}{theorem}
\newtheorem{lemma}[lemma]{Lemma}
\aliascntresetthe{lemma}
\newaliascnt{corollary}{theorem}
\newtheorem{corollary}[corollary]{Corollary}
\aliascntresetthe{corollary}
\theoremstyle{definition}
\newaliascnt{definition}{theorem}

\aliascntresetthe{definition}
\theoremstyle{remark}
\newaliascnt{remark}{theorem}
\newtheorem{remark}[remark]{Remark}
\aliascntresetthe{remark}

\crefname{theorem}{theorem}{theorems}
\Crefname{theorem}{Theorem}{Theorems}
\crefname{proposition}{proposition}{propositions}
\Crefname{proposition}{Proposition}{Propositions}
\crefname{lemma}{lemma}{lemmas}
\Crefname{lemma}{Lemma}{Lemmas}
\crefname{corollary}{corollary}{corollaries}
\Crefname{corollary}{Corollary}{Corollaries}
\crefname{definition}{definition}{definitions}
\Crefname{definition}{Definition}{Definitions}
\crefname{remark}{remark}{remarks}
\Crefname{remark}{Remark}{Remarks}

\newcommand{\A}{\mathbf A}
\newcommand{\Pj}{\mathbf P}
\newcommand{\F}{\mathbf F}
\newcommand{\Z}{\mathbf Z}

\newcommand{\cX}{\mathscr X}
\newcommand{\cU}{\mathscr U}
\newcommand{\cI}{\mathcal I}
\newcommand{\cR}{\mathcal R}
\newcommand{\BSing}{\operatorname{Sing}^{\mathrm B}}
\newcommand{\Spec}{\operatorname{Spec}}
\newcommand{\Proj}{\operatorname{Proj}}

\newcommand{\im}{\operatorname{im}}

\newcommand{\length}{\operatorname{length}}
\newcommand{\gdeg}{\operatorname{gdeg}}
\newcommand{\Prob}{\mathbf P}

\newcommand{\one}{\mathbf 1}
\newcommand{\ceil}[1]{\left\lceil #1\right\rceil}
\newcommand{\floor}[1]{\left\lfloor #1\right\rfloor}
\newcommand{\abs}[1]{\left|#1\right|}
\newcommand{\set}[1]{\left\{#1\right\}}
\newcommand{\restr}[1]{\big|_{#1}}
\newcommand{\orb}{\mathrm O}

\begin{document}
\let\WriteBookmarks\relax
\def\floatpagepagefraction{1}
\def\textpagefraction{.001}

\shorttitle{Frobenius-orbit slicing}
\shortauthors{Yutong Zhang and Yaoran Yang}

\title[mode=title]{Frobenius-orbit slicing and uniform elimination of positive-dimensional singular loci}

\author[1]{Yutong Zhang}
\cormark[1]
\ead{yutongzhang@stu.scu.edu.cn}
\credit{Conceptualization, Methodology, Formal analysis, Writing - original draft}

\author[1]{Yaoran Yang}
\ead{yangyaoran@stu.scu.edu.cn}
\credit{Formal analysis, Validation, Writing - review and editing}

\affiliation[1]{organization={School of Mathematics, Sichuan University},
            addressline={24 First Loop Road South Section I},
            city={Chengdu},
            postcode={610064},
            state={Sichuan},
            country={China}}

\cortext[1]{Corresponding author}

\begin{abstract}
Let $\cX\subseteq\Pj^n_{\Z}$ be a fixed integral quasiprojective subscheme, smooth over $\Z$ of relative dimension $r$. For each fixed $m\ge1$, we bound the probability that the $m$th principal-parts jet of the restriction of a uniform degree-$d$ form to $\cX_p$ has a positive-dimensional zero scheme. The bound is $C(d+1)^{N_m}p^{-\lambda_m(d)}$, where $N_m=\binom{r+m}{m}$ and $\lambda_m(d)=\floor{m(d+1)/(m+1)}$. For $m=1$, this gives the Bertini singular-locus estimate $C(d+1)^{r+1}p^{-\ceil{d/2}}$. It settles Poonen's arithmetic Bertini Conjecture~5.2 and, after increasing the degree threshold, yields $p^{-A}$ for every fixed $A>0$. For $c\le r$ independent hypersurfaces, the probability of a positive-dimensional Jacobian rank-degeneracy locus is bounded both by $C\sum_i(d_i+1)^{r+1}p^{-\ceil{d_i/2}}$ and by $C'(d_{\min}+1)^{r+1}p^{-\ceil{d_{\min}/2}}$. The proof uses filtered $Q$-adic decompositions, triangular normal Taylor blocks, and Jacobian-pivot charts of uniformly controlled complexity.
\end{abstract}


\maketitle

\section{Introduction}\label{sec:introduction}

Let $S_{d,p}=H^0(\Pj^n_{\F_p},\mathcal O_{\Pj^n}(d))$, equipped with the uniform distribution, and let $H_f\subseteq\Pj^n_{\F_p}$ be the hypersurface defined by $f\in S_{d,p}$. If $V$ is smooth of pure dimension $s\ge1$ over a field $k$ and $G$ is a regular function, or a local equation of a section of an invertible sheaf, define
\begin{equation}\label{eq:Bertini-singular-locus}
 \BSing_V(G)=V_V(G,dG)_{\mathrm{red}}.
\end{equation}
This definition is independent of the local trivialization. Indeed, if $G'=uG$ for a unit $u$, then $dG'=u\,dG+G\,du$, and hence $(G',dG')=(G,dG)$. The support of \eqref{eq:Bertini-singular-locus} is the locus where $V(G)$ is not smooth of relative dimension $s-1$ over $k$. It therefore contains every irreducible component of $V$ on which $G$ vanishes identically. This is the singular locus used in Poonen's arithmetic Bertini formulation, rather than only the intrinsic nonregular locus of $V(G)$. We write $\BSing(H_f\cap V)=\BSing_V(f|_V)$. If $s=0$, set $\BSing_V(G)=V(G)_{\mathrm{red}}$. Throughout, $\dim Z$ denotes the Krull dimension of $(Z_{\bar k})_{\mathrm{red}}$, and $\dim\varnothing=-\infty$.

Poonen asked in \cite[Conjecture~5.2]{Poonen2004} whether, for every integral quasiprojective $\cX\subseteq\Pj^n_{\Z}$ that is smooth over $\Z$ of relative dimension $r$, there is a constant $c>0$ such that
\begin{equation}\label{eq:Poonen-conjecture}
 \frac{\#\set{f\in S_{d,p}:\dim\BSing(H_f\cap\cX_p)\ge1}}{\#S_{d,p}}
 <\frac{c}{p^2}
\end{equation}
whenever $d$ and $p$ are sufficiently large. He also noted the heuristic expectation of a bound $c/p^k$ for every fixed integer $k\ge2$. The first-order estimate below is stronger: for every fixed real $A>0$, the probability in \eqref{eq:Poonen-conjecture} is at most $p^{-A}$ once the lower bound on $d$ is allowed to depend on $A$. Characteristic-$p$ derivative decoupling is particularly effective when $d/p$ is large, but its auxiliary degree remains bounded in the range $d\asymp p$. For projective arithmetic varieties, Wang proved a corresponding $p^{-2}$ estimate under smoothness and irreducibility hypotheses on the fibers \cite[Lemma~6.3]{Wang2022}. The quasiprojective case requires a uniform treatment of curves meeting the boundary.

We first state a higher-jet form of the estimate. Put $L_{d,p}=\mathcal O_{\Pj^n_{\F_p}}(d)|_{\cX_p}$. If $p_1,p_2:\cX_p\times_{\F_p}\cX_p\to\cX_p$ are the projections and $\mathcal I_\Delta$ is the ideal of the diagonal, set
\begin{equation}\label{eq:principal-parts-definition}
 \mathcal P^m_{\cX_p/\F_p}(L_{d,p})
 =p_{1*}\!\left(
 p_2^*L_{d,p}\otimes
 \mathcal O_{\cX_p\times_{\F_p}\cX_p}/\mathcal I_\Delta^{m+1}
 \right).
\end{equation}
This sheaf is locally free because $\cX_p/\F_p$ is smooth. Pullback along $p_2$ gives the canonical section
$j^m(f|_{\cX_p})\in H^0(\cX_p,\mathcal P^m_{\cX_p/\F_p}(L_{d,p}))$; let $J_m(f)=Z(j^m(f|_{\cX_p}))$. When $p>m$, \cref{lem:etale-taylor} identifies this zero scheme, after a local trivialization and a choice of \'{e}tale coordinates, with the common zero scheme of the divided coordinate derivatives of order at most $m$. The defining ideal is independent of the trivialization. Put
$N_m=\binom{r+m}{m}$ and $\lambda_m(d)=\floor{\frac{m(d+1)}{m+1}}$.

\begin{theorem}[Uniform higher-jet estimate]\label{thm:jet-global}
For every fixed $m\ge1$ there are constants $C,d_0,p_0$, depending only on the fixed embedding $\cX\subseteq\Pj^n_{\Z}$ and on $m$, such that for every prime $p\ge p_0$ and every $d\ge d_0$,
\begin{equation}\label{eq:jet-global}
 \Prob_{f\in S_{d,p}}\bigl(\dim J_m(f)\ge1\bigr)
 \le C(d+1)^{N_m}p^{-\lambda_m(d)}.
\end{equation}
\end{theorem}

\begin{corollary}[Poonen's conjecture]\label{cor:poonen}
There are constants $C,d_0,p_0$, depending only on the fixed embedding $\cX\subseteq\Pj^n_{\Z}$, such that
\begin{equation}\label{eq:first-order-global}
 \Prob_{f\in S_{d,p}}
 \bigl(\dim\BSing(H_f\cap\cX_p)\ge1\bigr)
 \le C(d+1)^{r+1}p^{-\ceil{d/2}}
\end{equation}
for $p\ge p_0$ and $d\ge d_0$. Consequently, \eqref{eq:Poonen-conjecture} holds. More generally, for every $A>0$, the left side of \eqref{eq:first-order-global} is at most $p^{-A}$ after increasing the degree threshold.
\end{corollary}

We next consider complete intersections. Fix $1\le c\le r$ and positive integers $d_1,\ldots,d_c$. For independent uniform forms $\mathbf f=(f_1,\ldots,f_c)$, let $Y_{\mathbf f}=\cX_p\cap H_{f_1}\cap\cdots\cap H_{f_c}$, and put $L_i=\mathcal O_{\Pj^n_{\F_p}}(d_i)|_{\cX_p}$ and $s_i=f_i|_{\cX_p}$. On $Y_{\mathbf f}$ there is a canonical homomorphism
\begin{equation}\label{eq:global-ci-differential-map}
 \Phi_{\mathbf f}:
 \bigoplus_{i=1}^cL_i^{-1}|_{Y_{\mathbf f}}
 \longrightarrow\Omega^1_{\cX_p/\F_p}|_{Y_{\mathbf f}}.
\end{equation}
If $e_i$ is a local frame of $L_i$ and $s_i=G_i e_i$, the $i$th summand sends $e_i^\vee$ to $dG_i|_{Y_{\mathbf f}}$. Under a change $e_i'=u_i e_i$, one has $G_i'=u_i^{-1}G_i$ and $dG_i'=u_i^{-1}dG_i$ on $Y_{\mathbf f}$, which is the transition rule for $L_i^{-1}$. Define $\Sigma_c(\mathbf f)=D_{c-1}(\Phi_{\mathbf f})$. Thus $\Sigma_c(\mathbf f)$ is locally defined by the $c\times c$ minors of a matrix for $\Phi_{\mathbf f}$, and its geometric points are characterized by
$G_1(x)=\cdots=G_c(x)=0$ and $\operatorname{rank}(dG_1(x),\ldots,dG_c(x))<c$.
By \cref{lem:ci-jacobian}, this locus is empty precisely when $Y_{\mathbf f}$ is empty or smooth of pure dimension $r-c$.

\begin{theorem}[Uniform complete-intersection estimate]\label{thm:ci-global}
There are constants $C,C',d_0,p_0$, depending only on the fixed embedding $\cX\subseteq\Pj^n_{\Z}$ and on $c$, such that, whenever $p\ge p_0$ and $d_{\min}=\min_i d_i\ge d_0$,
\begin{align}
 \Prob\bigl(\dim\Sigma_c(\mathbf f)\ge1\bigr)
 &\le C\sum_{i=1}^c(d_i+1)^{r+1}p^{-\ceil{d_i/2}},
 \label{eq:ci-sum}\\
 &\le C'(d_{\min}+1)^{r+1}p^{-\ceil{d_{\min}/2}}.
 \label{eq:ci-min}
\end{align}
\end{theorem}

\begin{remark}\label{rem:dmax-error}
A bound of the form $C(d_{\max}+1)^A p^{-\ceil{d_{\min}/2}}$ does not yield a fixed-power estimate that is uniform when $d_{\max}$ is unrestricted: with $p$ and $d_{\min}$ fixed, its right-hand side is unbounded as $d_{\max}\to\infty$. The termwise estimate \eqref{eq:ci-sum}, followed by the envelope argument giving \eqref{eq:ci-min}, avoids this problem.
\end{remark}

\subsection*{The orbit-slicing argument}

The local argument is carried out on an affine chart $U$ with distinct original dehomogenized coordinates $t_1=z_1,\ldots,t_r=z_r$ for which $(t_1,\ldots,t_r):U\to\A^r_{\F_q}$ is \'{e}tale. The fact that the $t_i$ are degree-one ambient variables is used in the filtered $Q$-adic decomposition. If $C$ is a curve contained in a positive-dimensional jet-zero locus, then some unsliced coordinate $t_j|_C$ is nonconstant. For a monic irreducible $Q\in\F_q[T]$ of degree $e$, one slices by $Q(t_j)=0$. A boundary-degree argument gives
\begin{equation}\label{eq:intro-hit}
 \Prob_Q\bigl(C\cap V(Q(t_j))=\varnothing\bigr)
 \ll\deg(\overline C)q^{-e}.
\end{equation}

The same equation has an algebraic role. The exact filtered expansion is
\begin{equation}\label{eq:intro-qadic}
 F=F_0+Q(t_j)F_1+\cdots+Q(t_j)^mF_m+\cdots.
\end{equation}
The blocks $F_a$ are independent and uniform in their respective filtered normal-form spaces. Modulo $Q(t_j)$, the first $m$ normal Hasse derivatives form a triangular system with diagonal term $Q'(t_j)^aF_a$ in order $a$. Every closed point of the slice has residue field containing $\F_{q^e}$, and the $a$th block contributes at least $\rho_a=\max\{0,\min\{e,d-ae+1\}\}$ independent conditions. Hence
\begin{equation}\label{eq:intro-Lambda}
 \Lambda_m(d,e)=\sum_{a=1}^m\rho_a
 =\min\set{me,d-e+1}.
\end{equation}
The recursion requires only that the miss factor $2b\Gamma_m(d)q^{-e}$ be bounded away from $1$. Since $\Gamma_m(d)$ is polynomial in $d$, this holds uniformly for $q\ge2$ when $e$ grows linearly with $d$. The quantity in \eqref{eq:intro-Lambda} is maximized at $e=\ceil{(d+1)/(m+1)}$, where it equals $\floor{m(d+1)/(m+1)}$. This optimality concerns the present choice of one common orbit degree and the first $m$ pure normal blocks.

For complete intersections, the equations are ordered by degree and introduced successively. Outside the preceding rank-degeneracy locus, a nonzero Jacobian minor defines a smooth pivot chart. Explicit tangent derivations on this chart reduce a new positive-dimensional rank defect to an ordinary hypersurface singular locus. The degree of the chart closure, the degree of its boundary, and the degrees of the tangent-derivation coefficients are polynomially controlled; the combined exponent in the local probability estimate is $r+1$.

Finite-field Bertini sieves have been developed for complete intersections \cite{BucurKedlaya2012}, semiample systems \cite{ErmanWood2015}, prescribed subschemes \cite{Poonen2008,Wutz2016,Gunther2017}, irreducibility \cite{CharlesPoonen2016}, motivic Taylor conditions \cite{BiluHowe2021}, and differential Taylor conditions \cite{Bertucci2026}. Random finite-field slicing appears in work on exceptional loci for Bertini irreducibility \cite{PoonenSlavov2022,KmenttShute2022}. Other recent results concern Hilbert--Samuel multiplicity \cite{AjitBertucci2026}, Galois-orbit interpolation and large singular loci \cite{AsgarliLoveYip2025}, and anti-Bertini embeddings \cite{ZhangYang2026}. A more detailed comparison is given in \cref{sec:literature}.

\subsection*{Organization}

\Cref{sec:charts} constructs the affine--\'{e}tale charts and proves the algebraic and geometric estimates used by orbit slicing. \Cref{sec:recursion} establishes the local recursion and globalizes it over the arithmetic family, proving \cref{thm:jet-global,cor:poonen}. \Cref{sec:ci} constructs the Jacobian-pivot charts, proves \cref{thm:ci-global}, and records consequences and comparisons in \cref{sec:consequences}.

\section{Local coordinates, filtered orbit slices, and geometric control}\label{sec:charts}

We first isolate geometric data that remain uniform in the residue characteristic. All degree bounds below are deliberately coarse; only polynomial dependence on $d$ is used.

\subsection{Affine coordinates and tangent derivations}

Let $R$ be a commutative ring. Fix a standard affine chart of $\Pj^n_R$ and write its dehomogenized coordinates as
$z_1,\ldots,z_n$, with $A_R=R[z_1,\ldots,z_n]$.
Throughout the affine arguments, $\deg$ denotes total degree and we use the convention
$\deg 0=-\infty$.
Whenever the zero polynomial occurs as a hypersurface equation in a B\'{e}zout estimate, it is regarded as the zero element of the explicitly prescribed positive-degree homogeneous piece. This convention leaves its zero scheme unchanged and makes every degree comparison below well defined.

If $U\subseteq\A^n_R$ is a locally closed affine subscheme and $D$ is a derivation of $\mathcal O(U)$, a \emph{polynomial representative of degree at most $a$} means an identity
\begin{equation}\label{eq:polynomial-derivation}
 D(F|_U)=\left(\sum_{\nu=1}^n A_\nu(z)\frac{\partial F}{\partial z_\nu}\right)\Big|_U
 \quad(F\in A_R),
 \qquad \deg A_\nu\le a.
\end{equation}
In particular, if $F\in A_R$ has total degree at most $d$, then the polynomial representative of $DF$ has degree
\begin{equation}\label{eq:derivative-degree}
 \deg(DF)\le d+a-1\le d+a.
\end{equation}
The insistence on the original dehomogenized coordinates is essential: uniform homogeneous forms dehomogenize to the full space $(A_{\F_p})_{\le d}$, and no auxiliary affine coordinates are introduced.

\begin{proposition}[Uniform controlled charts]\label{prop:controlled-charts}
Let $\cX\subseteq\Pj^n_{\Z}$ be integral, quasiprojective, and smooth over $\Z$ of relative dimension $r$. There exist a nonzero integer $N_0$, a finite affine open cover
$\cX_{\Z[1/N_0]}=\bigcup_{\alpha=1}^{M}\cU_\alpha$,
and integers
$\delta\ge1,\qquad b\ge1,\qquad \tau\ge1$
with the following properties. For each $\alpha$ there are:

\begin{enumerate}[label=\textup{(\roman*)},leftmargin=2.2em]
\item a standard projective chart $x_{a_\alpha}\ne0$ containing $\cU_\alpha$, with its original dehomogenized coordinates $z_1,\ldots,z_n$;
\item distinct coordinates, relabelled as $t_1=z_1,\ldots,t_r=z_r$, such that
$t_\alpha=(t_1,\ldots,t_r):\cU_\alpha\longrightarrow\A^r_{\Z[1/N_0]}$
is \'{e}tale;
\item derivations $D_{\alpha,1},\ldots,D_{\alpha,r}$ and units
$\Delta_{\alpha,1},\ldots,\Delta_{\alpha,r}\in\Gamma(\cU_\alpha,\mathcal O)^\times$ satisfying
$D_{\alpha,i}(t_j)=\Delta_{\alpha,i}\delta_{ij}$;
\item polynomial representatives in the same variables $z_1,\ldots,z_n$ for every $D_{\alpha,i}$, with coefficient degrees at most $\tau$;
\item an integral scheme-theoretic projective closure
$\overline{\cU}_\alpha\subseteq\Pj^n_{\Z[1/N_0]}$ whose nonempty geometric fibers are equidimensional of dimension $r$ and satisfy
\begin{equation}\label{eq:closure-component-degree}
 \sum_{Z\in\operatorname{Irr}((\overline{\cU}_{\alpha,\bar s})_{\mathrm{red}})}
 \deg Z\le\delta,
\end{equation}
for every geometric point $\bar s$ of $\Spec\Z[1/N_0]$, together with a homogeneous polynomial
$B_\alpha\in\Z[1/N_0][x_0,\ldots,x_n]$ of degree at most $b$, whose coefficients generate the unit ideal in $\Z[1/N_0]$, such that
\begin{equation}\label{eq:boundary-principal}
 \cU_\alpha=\overline{\cU}_\alpha\cap D_+(B_\alpha).
\end{equation}
\end{enumerate}
The \'{e}tale maps, unit and derivation identities, principal-open descriptions, and geometric-fiber degree bounds are preserved by base change; in particular, each base change to $\F_p$ for $p\nmid N_0$ has the stated local data. No integrality of an individual fiber is asserted. We do not assert that $\overline{\cU}_{\alpha,\bar s}$ is the scheme-theoretic closure of $\cU_{\alpha,\bar s}$ inside the fiber; the argument uses only the base-changed principal-open identity and the uniform bound \eqref{eq:closure-component-degree}.
\end{proposition}

\begin{proof}
Intersect $\cX$ with the standard affine charts of $\Pj^n_{\Z}$. For a locally closed immersion into $\A^n_{\Z}$, the ambient differentials $dz_1,\ldots,dz_n$ generate $\Omega^1_{\cX/\Z}$. Since $\cX/\Z$ is smooth of relative dimension $r$, the opens on which one of the wedges
$dz_{i_1}\wedge\cdots\wedge dz_{i_r}$
is a basis cover $\cX$. On each standard affine chart, $\cX$ is open in its scheme-theoretic affine closure. Distinguished opens in that closure form a basis, so quasicompactness gives finitely many principal affine opens $\cU_\alpha$ contained in the corresponding wedge-basis loci. After relabelling, $dt_1,\ldots,dt_r$ is a basis on $\cU_\alpha$, with $t_i=z_i$.

For
$t=(t_1,\ldots,t_r):\cU_\alpha\longrightarrow\A^r_{\Z}$,
the map $t^*\Omega^1_{\A^r_{\Z}/\Z}\to\Omega^1_{\cU_\alpha/\Z}$ is an isomorphism. Both source and target of $t$ are smooth over $\Z$. The transitivity triangle for cotangent complexes therefore gives $L_{\cU_\alpha/\A^r_{\Z}}=0$; since $t$ is locally of finite presentation, the infinitesimal criterion for \'{e}taleness shows that $t$ is \'{e}tale. Equivalently, the dual basis gives derivations
$\partial_1,\ldots,\partial_r$ satisfying $\partial_i(t_j)=\delta_{ij}$.

Write $\cU_\alpha=D(g_\alpha)$ in its scheme-theoretic affine closure inside the fixed standard affine space. Each regular function $\partial_i(z_\nu)$ lies in the localization by $g_\alpha$. Choose one exponent $N_\alpha$ clearing all of their denominators and set
$D_{\alpha,i}=g_\alpha^{N_\alpha}\partial_i$ and $\Delta_{\alpha,i}=g_\alpha^{N_\alpha}$.
After lifting the finitely many numerators to $\Z[z_1,\ldots,z_n]$, the values $D_{\alpha,i}(z_\nu)$ have polynomial representatives in the original ambient coordinates. A derivation is determined by its values on these generators, so \eqref{eq:polynomial-derivation} holds, and
$D_{\alpha,i}(t_j)=\Delta_{\alpha,i}\delta_{ij}$
with $\Delta_{\alpha,i}$ a unit on $D(g_\alpha)$. Taking the maximum of the finitely many coefficient degrees gives $\tau$.

Let $X^{\mathrm{aff}}_\alpha$ denote the scheme-theoretic affine closure in the fixed standard affine space, so that $\cU_\alpha=D_{X^{\mathrm{aff}}_\alpha}(g_\alpha)$. Since $X^{\mathrm{aff}}_\alpha$ is integral and $\cU_\alpha$ is nonempty, $g_\alpha$ is nonzero and $\cU_\alpha$ is dense in $X^{\mathrm{aff}}_\alpha$. Let $\overline{\cU}_\alpha$ be the scheme-theoretic closure of $\cU_\alpha$ in the same $\Pj^n_{\Z}$; it is integral, and its intersection with the standard affine chart is $X^{\mathrm{aff}}_\alpha$. Choose a polynomial lift $\widetilde g_\alpha\in\Z[z_1,\ldots,z_n]$ of $g_\alpha$, let $g_\alpha^{\mathrm{hom}}$ be its homogenization with respect to $x_{a_\alpha}$, and put
$B_\alpha=x_{a_\alpha}g_\alpha^{\mathrm{hom}}$.
Then
$D_+(B_\alpha)=D_+(x_{a_\alpha})\cap D_+(g_\alpha^{\mathrm{hom}})$,
and consequently
$\cU_\alpha=\overline{\cU}_\alpha\cap D_+(B_\alpha)$
as schemes. The finite list gives a common degree bound $b$. After enlarging the integer inverted below by the contents of the finitely many $B_\alpha$, their coefficients generate the unit ideal over the resulting base, so no $B_\alpha$ becomes the zero polynomial on a remaining fiber.

After inverting one integer $N_0$, generic flatness makes every
$\overline{\cU}_\alpha$ flat and projective over
$S=\Spec\Z[1/N_0]$, and all the preceding data and the finite cover are defined over $S$.
The generic fiber is integral of dimension $r$, so the integral total space has dimension
$r+1$. Let $s\in S$ be a closed point and let $\xi$ be the generic point of an irreducible
component of the fiber over $s$. On an affine neighborhood of $\xi$, a uniformizer $\pi_s$
of the regular one-dimensional base is a nonzerodivisor by flatness, and the prime
corresponding to $\xi$ is minimal over $(\pi_s)$. The principal ideal theorem gives height at
most one, while integrality and $\pi_s\ne0$ give height at least one; hence the height is exactly
one. Since schemes of finite type over the excellent, hence universally catenary, scheme $S$
satisfy the dimension formula, applied to $\xi$ over $s$ it gives
\[
 \operatorname{ht}_{\overline{\cU}_\alpha}(\xi)
 +\operatorname{trdeg}_{\kappa(s)}\kappa(\xi)
 =\operatorname{ht}_S(s)
 +\operatorname{trdeg}_{K(S)}
   K(\overline{\cU}_\alpha)
 =1+r.
\]
Since $\operatorname{ht}_{\overline{\cU}_\alpha}(\xi)=1$, one has
$\operatorname{trdeg}_{\kappa(s)}\kappa(\xi)=r$. Thus every irreducible
component of every ordinary closed fiber has dimension $r$. Together with the integral
generic fiber of dimension $r$, this says that
$\overline{\cU}_\alpha\to S$ has relative dimension $r$: every nonempty ordinary fiber is
equidimensional of dimension $r$. Relative dimension is preserved by arbitrary base change
\cite[Tag 02NI]{StacksProject}. Hence every nonempty geometric fiber is equidimensional of dimension $r$.

Projective flatness makes the Hilbert polynomial with respect to $\mathcal O(1)$ locally
constant on $S$, hence constant because $S$ is connected. Since every geometric fiber is
equidimensional of dimension $r$, the leading coefficient of this Hilbert polynomial is the
degree of its fundamental $r$-cycle. Let $D_\alpha$ denote this common degree. For every geometric point $\bar s$, if $\eta_Z$
denotes the generic point of an irreducible component
$Z\in\operatorname{Irr}((\overline{\cU}_{\alpha,\bar s})_{\mathrm{red}})$, then
\[
 D_\alpha
 =\sum_{Z\in\operatorname{Irr}((\overline{\cU}_{\alpha,\bar s})_{\mathrm{red}})}
 \length\!\left(\mathcal O_{\overline{\cU}_{\alpha,\bar s},\eta_Z}\right)
 \deg Z.
\]
Every displayed length is a positive integer. Therefore the unweighted sum of the degrees
of the reduced irreducible components is at most $D_\alpha$. Taking
$\delta=\max_\alpha D_\alpha$ proves \eqref{eq:closure-component-degree}.
Units, \'{e}taleness, the principal-open identity, and the polynomial-representative identities
are all stable under base change. See \cite[Tags 02GU, 039P, 01UA, 02NI]{StacksProject} and
\cite[IV,~\S17]{EGAIV}.
\end{proof}

\subsection{The differential singular scheme}

Fix a finite field $k=\F_q$ and one chart $U=\cU_{\alpha,k}$. Suppress $\alpha$ from the notation. Thus
$t=(t_1,\ldots,t_r):U\to\A^r_k$
is \'{e}tale, and there are derivations $D_i$ satisfying $D_i(t_j)=\Delta_i\delta_{ij}$, where $\Delta_i\in\Gamma(U,\mathcal O_U)^\times$.
For an affine polynomial $F$, define
$\Sigma_U(F)=V_U(F,D_1F,\ldots,D_rF)$.

\begin{lemma}[Jacobian description]\label{lem:jacobian-description}
Assume $r\ge1$. The support of $\Sigma_U(F)$ is exactly the locus at which $V_U(F)$ is not smooth of relative dimension $r-1$ over $k$. In particular, for a homogeneous form $f$ and its dehomogenization $F$ on the chart,
$\BSing(H_f\cap U)=\bigl(\Sigma_U(F)\bigr)_{\mathrm{red}}$.
\end{lemma}

\begin{proof}
It is enough to test the assertion after extending the ground field to an algebraic closure and at geometric points $x\in V_U(F)$. Since $U/k$ is smooth of relative dimension $r$, its cotangent space at $x$ has basis
$dt_1(x),\ldots,dt_r(x)$, and the operators $D_i/\Delta_i$ are dual to this basis. Hence
$dF(x)=0 \quad\Longleftrightarrow\quad D_1F(x)=\cdots=D_rF(x)=0$.
The smooth-hypersurface criterion in a smooth $k$-scheme says that $V_U(F)$ is smooth of relative dimension $r-1$ at $x$ exactly when $dF(x)\ne0$. If $F$ vanishes identically on a component of $U$, then $dF$ vanishes on that component and the section has the wrong relative dimension there, so the same criterion includes the whole component. The equality of reduced closed subschemes follows from the definition of $\BSing_U(F)$.
\end{proof}

\subsection{Higher coordinate derivatives}

Put
$\partial_i=\Delta_i^{-1}D_i$.
These derivations are the coordinate derivations dual to $dt_1,\ldots,dt_r$; they satisfy
$\partial_i(t_j)=\delta_{ij}$.
Moreover, $[\partial_i,\partial_j]$ annihilates every $t_\ell$, hence every differential in the basis $dt_1,\ldots,dt_r$; therefore $[\partial_i,\partial_j]=0$, so the coordinate derivations commute.
For a multi-index $\gamma=(\gamma_1,\ldots,\gamma_r)$ such that every $\gamma_i!$ is invertible in $k$, write
\[
 \partial^{[\gamma]}
 =\frac{\partial_1^{\gamma_1}\cdots\partial_r^{\gamma_r}}
 {\gamma_1!\cdots\gamma_r!}.
\]
In every later use one has $\abs\gamma\le m$ and $\operatorname{char}k>m$, so this divided derivative is defined.

\begin{lemma}[Uniform polynomial representatives for higher derivatives]\label{lem:higher-representatives}
Fix $m\ge1$ and one controlled integral chart $\cU_\alpha$ from \cref{prop:controlled-charts}. After base change to
$R_m=\Z[1/(N_0m!)]$,
there is a constant $\tau_m$, depending only on this integral chart and on $m$, with the following property. For every field $k$ over $R_m$, every multi-index $\gamma$ with $\abs\gamma\le m$, and every polynomial
$F\in k[z_1,\ldots,z_n]$ of degree at most $d$, there are
\begin{itemize}[leftmargin=2em]
\item a unit $u_\gamma\in\Gamma(U,\mathcal O_U)^\times$, obtained by base change from a unit on $\cU_{\alpha,R_m}$ and independent of $F$; and
\item a polynomial $P_{\gamma,F}\in k[z_1,\ldots,z_n]$
\end{itemize}
such that
\begin{equation}\label{eq:higher-representative}
 P_{\gamma,F}|_U=u_\gamma\partial^{[\gamma]}F,
 \qquad
 \deg P_{\gamma,F}\le d+\tau_m.
\end{equation}
The same value of $\tau_m$ works after every base change from $R_m$.
\end{lemma}

\begin{proof}
Work first over $R_m$. Choose polynomial representatives
$a_{i\nu}(z)\in R_m[z_1,\ldots,z_n] \qquad(1\le i\le r,\ 1\le\nu\le n)$
for the coefficients of $D_i$, so that on $\cU_{\alpha,R_m}$ one has
$D_i=\sum_{\nu=1}^n a_{i\nu}(z)\frac{\partial}{\partial z_\nu}$.
Because $t_i=z_i$ and $D_i(t_i)=\Delta_i$, the polynomial
$\delta_i:=a_{ii}$ represents the unit $\Delta_i$. Put
$\delta=\prod_{i=1}^r\delta_i$.
Its restriction to the chart is a unit. On restrictions of ambient polynomials,
\begin{equation}\label{eq:coordinate-operator-ambient}
 \partial_i
 =\Delta_i^{-1}D_i
 =\sum_{\nu=1}^n\frac{a_{i\nu}(z)}{\delta_i(z)}
   \frac{\partial}{\partial z_\nu}.
\end{equation}

We claim that, for every multi-index $\gamma$ with $\abs\gamma\le m$, there are an integer $M_\gamma\ge0$ and polynomials
$c_{\gamma,\eta}(z)\in R_m[z_1,\ldots,z_n]$, indexed by
$\abs\eta\le\abs\gamma$, such that on restrictions of ambient polynomials
\begin{equation}\label{eq:ambient-differential-operator-expansion}
 \partial^\gamma
 =\sum_{\abs\eta\le\abs\gamma}
   \frac{c_{\gamma,\eta}(z)}{\delta(z)^{M_\gamma}}
   \partial_z^\eta.
\end{equation}
Moreover, the degrees of all $c_{\gamma,\eta}$ are bounded by a constant depending only on the chart and on $\gamma$. For $\gamma=0$ this is immediate. Suppose the assertion holds for $\gamma$ and apply $\partial_i$. Using \eqref{eq:coordinate-operator-ambient}, the ordinary product rule, and
\[
 \frac{\partial}{\partial z_\nu}\bigl(\delta^{-M_\gamma}\bigr)
 =-M_\gamma\delta^{-M_\gamma-1}
   \frac{\partial\delta}{\partial z_\nu},
\]
we obtain an expression of the same form for $\partial_i\partial^\gamma$ after multiplying the denominator by a fixed additional power of $\delta$. Every new numerator is a finite sum of products of the previously constructed numerators, the fixed polynomials $a_{i\nu}$, $\delta$, and their first ordinary derivatives. Its degree therefore increases by a bounded amount depending only on the controlled chart. This proves the claim by induction on $\abs\gamma$.

Since $m!$ is invertible in $R_m$, divide \eqref{eq:ambient-differential-operator-expansion} by
$\gamma!=\gamma_1!\cdots\gamma_r!$. Set
$u_\gamma=\delta^{M_\gamma}|_{\cU_{\alpha,R_m}}$
and, for a polynomial $F$, define
\[
 P_{\gamma,F}
 =\frac1{\gamma!}
  \sum_{\abs\eta\le\abs\gamma}
  c_{\gamma,\eta}(z)\,\partial_z^\eta F.
\]
Then $P_{\gamma,F}|_U=u_\gamma\partial^{[\gamma]}F$ after every base change. If $C_\gamma$ bounds the degrees of the coefficients $c_{\gamma,\eta}$, then
\[
 \deg\bigl(c_{\gamma,\eta}\partial_z^\eta F\bigr)
 \le C_\gamma+d-\abs\eta
 \le d+C_\gamma.
\]
Taking $\tau_m=\max_{\abs\gamma\le m}C_\gamma$ proves both the degree estimate and its uniformity under base change.
\end{proof}

\begin{lemma}[\'{E}tale Taylor coordinates and principal parts]\label{lem:etale-taylor}
Fix $m\ge0$ and assume that $\operatorname{char}k=0$ or $\operatorname{char}k>m$. Let
$\mathfrak a=(\xi_1,\ldots,\xi_r) \subseteq\Gamma(U,\mathcal O_U)[\xi_1,\ldots,\xi_r]$.
Then
\begin{equation}\label{eq:etale-Taylor-homomorphism}
 \Theta^{(m)}(G)
 =\sum_{\abs\gamma\le m}\partial^{[\gamma]}G\,\xi^\gamma
 \pmod{\mathfrak a^{m+1}}
\end{equation}
defines a $k$-algebra homomorphism
\[
 \Theta^{(m)}:\Gamma(U,\mathcal O_U)
 \longrightarrow
 \Gamma(U,\mathcal O_U)[\xi_1,\ldots,\xi_r]/\mathfrak a^{m+1},
\]
uniquely characterized by
$\Theta^{(m)}(t_i)=t_i+\xi_i$
and by the condition that its reduction modulo $\mathfrak a$ is the identity on $\Gamma(U,\mathcal O_U)$.
Under the canonical identification of the $m$th infinitesimal neighborhood of the diagonal with
\begin{equation}\label{eq:principal-parts-etale-coordinates}
 p_{1*}\!\left(
 \mathcal O_{U\times_kU}/\mathcal I_\Delta^{m+1}
 \right)
 \cong
 \mathcal O_U[\xi_1,\ldots,\xi_r]/(\xi_1,\ldots,\xi_r)^{m+1},
\end{equation}
where $\xi_i=p_2^*t_i-p_1^*t_i$, the principal-parts jet of a function $G$ is represented by the right side of \eqref{eq:etale-Taylor-homomorphism}. Consequently, for every unit $u\in\Gamma(U,\mathcal O_U)^\times$,
\begin{equation}\label{eq:jet-unit-invariance}
 \bigl(\partial^{[\gamma]}(uG):\abs\gamma\le m\bigr)
 =
 \bigl(\partial^{[\gamma]}G:\abs\gamma\le m\bigr)
\end{equation}
as ideal sheaves on $U$.
\end{lemma}

\begin{proof}
Write $B=\Gamma(U,\mathcal O_U)$ and $R=k[t_1,\ldots,t_r]$. Since
$U\to\A^r_k$ is \'{e}tale, $B$ is an \'{e}tale, hence formally \'{e}tale, $R$-algebra. The map
\[
 R\longrightarrow B[\xi_1,\ldots,\xi_r]/\mathfrak a^{m+1},
 \qquad
 t_i\longmapsto t_i+\xi_i,
\]
reduces modulo $\mathfrak a$ to the given structure map $R\to B$. Formal \'{e}taleness therefore gives a unique lift
$B\longrightarrow B[\xi_1,\ldots,\xi_r]/\mathfrak a^{m+1}$
reducing to the identity modulo $\mathfrak a$.

Because the $\partial_i$ commute and all factorials up to $m$ are invertible, the generalized Leibniz rule shows directly that the truncated exponential
\[
 G\longmapsto
 \sum_{\abs\gamma\le m}\partial^{[\gamma]}G\,\xi^\gamma
 \pmod{\mathfrak a^{m+1}}
\]
is a $k$-algebra homomorphism. It sends $t_i$ to $t_i+\xi_i$, so uniqueness identifies it with the formally \'{e}tale lift. For completeness, put
$C=(B\otimes_kB)/I^{m+1}$ and $I=\ker(B\otimes_kB\xrightarrow{\mathrm{mult}}B)$.
As a $B$-algebra through the first tensor factor, $C$ receives the homomorphism
\[
 B[\xi_1,\ldots,\xi_r]/\mathfrak a^{m+1}\longrightarrow C,
 \qquad
 \xi_i\longmapsto 1\otimes t_i-t_i\otimes1.
\]
Conversely, send the first tensor factor $B$ identically to $B$ and the second tensor factor through $\Theta^{(m)}$; because $I$ maps into $\mathfrak a$, this induces a homomorphism
$C\longrightarrow B[\xi_1,\ldots,\xi_r]/\mathfrak a^{m+1}$.
The two composites are the identity on the first factor and on every $\xi_i$. For the second tensor factor, equip $C$ with the $R$-algebra structure $t_i\mapsto1\otimes t_i$. The remaining identity follows from formal \'{e}taleness: both maps $B\to C$ are $R$-algebra maps and reduce modulo $I$ to the identity. Thus the maps are inverse. Sheafifying gives \eqref{eq:principal-parts-etale-coordinates}, and under this identification $p_2^*G$ is exactly the Taylor expansion \eqref{eq:etale-Taylor-homomorphism}.

Finally, the Hasse product rule gives
\[
 \partial^{[\gamma]}(uG)
 =u\,\partial^{[\gamma]}G
 +\sum_{\substack{\eta+\nu=\gamma\\\nu\ne\gamma}}
 \partial^{[\eta]}u\,\partial^{[\nu]}G.
\]
Ordering multi-indices by total degree shows that the generators on the left of \eqref{eq:jet-unit-invariance} are obtained from those on the right by an upper-triangular matrix with diagonal entry $u$. Hence the left ideal is contained in the right ideal; applying the same argument to $u^{-1}$ gives the reverse inclusion.
\end{proof}

\subsection{Filtered polynomial quotients and iterated orbit slices}\label{sec:filtered}

The recursion repeatedly restricts a random degree-$d$ polynomial to slices of the form $Q_i(t_i)=0$. Uniformity must therefore be formulated in the filtered quotient by the slice equations rather than in the original polynomial space.

Let
$A=k[z_1,\ldots,z_n]$ and $A_{\le d}=\set{F\in A:\deg F\le d}$,
using the original dehomogenized coordinates of the standard projective chart containing $U$. Fix an integer $e\ge1$. For $I\subseteq[r]=\set{1,\ldots,r}$ and monic degree-$e$ polynomials
$\boldsymbol Q_I=(Q_i)_{i\in I}$, with $Q_i\in k[T]$,
put
$J_I=\bigl(Q_i(t_i):i\in I\bigr)\subseteq A$.
Define
\begin{equation}\label{eq:normal-form-space}
 \cR_I(d)=
 \operatorname{span}_k\set{
 z_1^{a_1}\cdots z_n^{a_n}:
 \sum_{\nu=1}^na_\nu\le d,\quad
 a_i<e\text{ for }i\in I
 }.
\end{equation}
For $d<0$, set $\cR_I(d)=0$.

\begin{lemma}[Filtered normal form and the full $Q$-adic expansion]\label{lem:filtered-normal-form}
For every $d\ge0$, successive reduction by the monic polynomials $Q_i(t_i)$ induces a vector-space isomorphism
\begin{equation}\label{eq:filtered-quotient-isomorphism}
 \cR_I(d)\xrightarrow{\sim}
 \im\left(A_{\le d}\longrightarrow A/J_I\right).
\end{equation}
Equivalently, every $F\in A_{\le d}$ has a unique congruent representative
$\operatorname{NF}_I(F)\in\cR_I(d)$.
If $j\notin I$ and $Q_j$ is monic of degree $e$, then
\begin{equation}\label{eq:filtered-direct-sum}
 \cR_I(d)
 =\cR_{I\cup\set{j}}(d)
 \oplus Q_j(t_j)\cR_I(d-e),
\end{equation}
and, after iterating this decomposition,
\begin{equation}\label{eq:q-adic-sum}
 \cR_I(d)=
 \bigoplus_{a=0}^{\floor{d/e}}
 Q_j(t_j)^a\cR_{I\cup\set{j}}(d-ae).
\end{equation}
\end{lemma}

\begin{proof}
Because the generators of $J_I$ are monic univariate polynomials in distinct variables, successive Euclidean division is independent of the order and produces a unique remainder whose exponent in every variable indexed by $I$ is less than $e$. Replacing a power $t_i^h$, $h\ge e$, by the lower terms of $Q_i$ strictly decreases that exponent and never increases total degree. Hence a polynomial of total degree at most $d$ reduces into $\cR_I(d)$.

A polynomial in $\cR_I(d)$ that belongs to $J_I$ has zero multivariate remainder, so it is zero. This proves \eqref{eq:filtered-quotient-isomorphism}. Dividing an element of $\cR_I(d)$ by $Q_j(t_j)$ gives a remainder in $\cR_{I\cup\set{j}}(d)$ and a quotient in $\cR_I(d-e)$; uniqueness of univariate division gives \eqref{eq:filtered-direct-sum}. Repeated division of the quotient gives \eqref{eq:q-adic-sum}, and uniqueness gives directness.
\end{proof}

\begin{corollary}[Exact preservation of uniform randomness]\label{cor:uniform-randomness}
Let $F$ be uniform in $\cR_I(d)$. For fixed $j\notin I$ and fixed monic $Q_j$ of degree $e$, write
$F=\sum_{a=0}^{\floor{d/e}}Q_j(t_j)^aF_a$
using \eqref{eq:q-adic-sum}. Then the nonzero blocks $F_a$ are independent, and
\begin{equation}\label{eq:qadic-block-distribution}
 F_a\sim\operatorname{Unif}\bigl(\cR_{I\cup\set{j}}(d-ae)\bigr).
\end{equation}
In particular, the one-step decomposition $F=F_0+Q_j(t_j)H$ has independent factors
\[
 F_0\sim\operatorname{Unif}\bigl(\cR_{I\cup\set{j}}(d)\bigr),
 \qquad
 H\sim\operatorname{Unif}\bigl(\cR_I(d-e)\bigr).
\]
\end{corollary}

\begin{proof}
Uniform measure on a finite direct sum is the product of the uniform measures on its summands.
\end{proof}

\subsubsection{Iterated orbit slices and jet-zero loci}

Assume from now on that every $Q_i$ is monic irreducible of degree $e$. Define
$Y_I=U\cap\bigcap_{i\in I}V\bigl(Q_i(t_i)\bigr)$.
Since finite fields are perfect, every $Q_i$ is separable. If
$T_I=V_{\A^r_k}\bigl(Q_i(t_i):i\in I\bigr)$,
then the projection of $T_I$ to the unsliced coordinate space
$\A^{r-\abs I}_k$ is finite \'{e}tale. Moreover,
$Y_I=U\times_{\A^r_k}T_I$.
Thus $Y_I$ is smooth, every nonempty irreducible component has dimension $r-\abs I$, and the unsliced functions $t_j$, $j\notin I$, form an \'{e}tale coordinate system.

Fix $m\ge1$ and assume $\operatorname{char}k>m$. For $F\in\cR_I(d)$ define
\begin{equation}\label{eq:WI-m}
 W_I^{(m)}(F)=Y_I\cap
 V\bigl(\partial^{[\gamma]}F:
 \gamma_i=0\ (i\in I),\quad \abs\gamma\le m\bigr).
\end{equation}

\begin{lemma}[Jet interpretation and tangential well-definedness]\label{lem:WI-interpretation}
The scheme $W_I^{(m)}(F)$ depends only on the class of $F$ modulo $J_I$, and it is the zero scheme of the $m$th principal-parts jet of $F|_{Y_I}$. In particular,
$W_\varnothing^{(m)}(F)=Z\bigl(j^m(F|_U)\bigr)$
as zero schemes in the chosen trivialization. For $m=1$,
\begin{equation}\label{eq:W1-singular}
 W_\varnothing^{(1)}(F)=\Sigma_U(F),
 \qquad
 \bigl(W_\varnothing^{(1)}(F)\bigr)_{\mathrm{red}}
 =\BSing_U(F).
\end{equation}
Finally,
\begin{equation}\label{eq:terminal-dimension}
 \dim W_{[r]}^{(m)}(F)\le0.
\end{equation}
\end{lemma}

\begin{proof}
If $i\in I$ and $\gamma_i=0$, the Hasse product rule gives
$\partial^{[\gamma]}\bigl(Q_i(t_i)G\bigr) =Q_i(t_i)\partial^{[\gamma]}G$,
because every positive Hasse derivative of $Q_i(t_i)$ is supported in the $i$th direction. Hence all equations in \eqref{eq:WI-m} are unchanged on $Y_I$ after modifying $F$ by an element of $J_I$.

The unsliced functions $t_j$, $j\notin I$, form an \'{e}tale coordinate system on $Y_I$, and their coordinate derivations are the restrictions of the corresponding $\partial_j$ because those derivations preserve the slice ideal. Applying \cref{lem:etale-taylor} to $Y_I$ identifies its principal-parts algebra of order $m$ with the truncated polynomial algebra in the unsliced coordinate differences, and identifies the coefficients of the jet of $F|_{Y_I}$ with the functions
\[
 \partial^{[\gamma]}F\restr{Y_I},
 \qquad
 \gamma_i=0\ (i\in I),\quad \abs\gamma\le m.
\]
Thus the defining ideal in \eqref{eq:WI-m} is exactly the ideal of the jet-zero scheme, proving the scheme-theoretic assertion. For $m=1$, the ideals
$(F,\partial_1F,\ldots,\partial_rF)$ and $(F,D_1F,\ldots,D_rF)$
coincide because every $\Delta_j$ is a unit. Hence $W_\varnothing^{(1)}(F)=\Sigma_U(F)$, and the reduced equality follows from \cref{lem:jacobian-description}. When $I=[r]$, the base $T_I$ and hence $Y_I$ are zero-dimensional, proving \eqref{eq:terminal-dimension}.
\end{proof}

\subsection{Geometric control, orbit slices, and normal entropy}\label{sec:degree}

All geometric statements in this section are taken after base change to an algebraic closure $\bar k$. For a projective closed subscheme $T\subseteq\Pj^n_{\bar k}$, define
$\gdeg(T)=\sum_{C\in\operatorname{Irr}(T_{\mathrm{red}})}\deg C$.
If $Z$ is a quasiprojective closed subscheme of $U_{\bar k}$, define
$\gdeg(Z)=\gdeg(\overline Z)$,
where $\overline Z$ is its projective closure in $\overline U_{\bar k}$. This mixed-dimensional geometric degree bounds the degree of each irreducible component and, when $Z$ is zero-dimensional, bounds the number of geometric points.

\begin{lemma}[Closure through a principal boundary]\label{lem:principal-open-closure}
Let $\overline U\subseteq\Pj^n_{\bar k}$ be projective, let
$U=\overline U\cap D_+(B)$, and let $T\subseteq\overline U$ be closed. Then the projective closure of $(T\cap U)_{\mathrm{red}}$ is the union of those irreducible components of $T_{\mathrm{red}}$ that are not contained in $V(B)$. In particular,
$\gdeg(T\cap U)\le\gdeg(T)$.
\end{lemma}

\begin{proof}
Write $T_{\mathrm{red}}=\bigcup_\lambda T_\lambda$ as the union of its irreducible components. If $T_\lambda\subseteq V(B)$, then $T_\lambda\cap U$ is empty. Otherwise $T_\lambda\cap U$ is a nonempty dense open subset of $T_\lambda$, so its closure is $T_\lambda$. Taking the union proves both assertions. Equivalently, passing from the projective equations to the closure of their zero locus on $U$ amounts, on reduced supports, to deleting the components killed by saturation with respect to $B$.
\end{proof}

\begin{lemma}[Mixed-dimensional B\'{e}zout inequality]\label{lem:mixed-bezout}
Let $T\subseteq\Pj^n_{\bar k}$ be projective and let $G$ be a homogeneous form of degree $m\ge1$. Then
\begin{equation}\label{eq:mixed-bezout}
 \gdeg\bigl((T\cap V(G))_{\mathrm{red}}\bigr)
 \le m\,\gdeg(T).
\end{equation}
The same conclusion holds for the zero form when it is regarded as a form of the prescribed degree $m$.
\end{lemma}

\begin{proof}
It is enough to argue separately on each irreducible component $Z$ of $T_{\mathrm{red}}$. If $G$ vanishes identically on $Z$, then $Z$ remains in the intersection and
$\deg Z\le m\deg Z$. If $G$ does not vanish identically and $\dim Z=0$, then $Z\cap V(G)$ is empty. If $G$ does not vanish identically and $\dim Z\ge1$, the proper hypersurface-section cycle on $Z$ has degree $m\deg Z$, so the sum of the degrees of the irreducible components of its reduced support is at most $m\deg Z$. Components arising from different $Z$ may coincide, which only lowers the degree of the union. Summing over $Z$ proves the claim. If $G=0$, the left side is $\gdeg(T)\le m\gdeg(T)$.
\end{proof}

\subsubsection{Uniform degree bounds for recursive jet loci}

For $s\ge0$ put
$N_m(s)=\binom{s+m}{m}$ and $N_m=N_m(r)$.

\begin{lemma}[Degree of all recursive jet loci]\label{lem:degree-WI-m}
Fix $m\ge1$ and assume $\operatorname{char}k>m$. There is a constant $C_m$, depending only on the fixed chart and $m$, such that, whenever $d\ge1$, $1\le e\le d$, $I\subseteq[r]$, and $F\in\cR_I(d)$,
\begin{align}
 \gdeg\bigl(W_I^{(m)}(F)\bigr)
 &\le \delta e^{\abs I}(d+\tau_m)^{N_m(r-\abs I)},
 \label{eq:gdeg-WI-refined}\\
 &\le C_m(d+1)^{N_m}.
 \label{eq:gdeg-WI-m}
\end{align}
Consequently, if $W_I^{(m)}(F)$ is zero-dimensional, then its reduced support has at most
\begin{equation}\label{eq:number-points-zero}
 \Gamma_m(d):=C_m(d+1)^{N_m}
\end{equation}
geometric points.
\end{lemma}

\begin{proof}
Work in the fixed projective closure $\overline U_{\bar k}$. Homogenize each slice equation $Q_i(t_i)$ to degree $e$. By \cref{lem:higher-representatives}, the $N_m(r-\abs I)$ jet equations have polynomial representatives of degree at most $d+\tau_m$, up to multiplication by units on $U$. Since $d\ge1$, the prescribed degree $d+\tau_m$ is positive. Homogenize every representative to this common degree with respect to the standard-chart coordinate; for a representative of smaller degree this only introduces a power of that coordinate, and the zero representative is regarded as the zero form of the same prescribed degree. These operations do not change the zero scheme on $U$. Let $T_I^{(m)}(F)$ be the common projective zero locus of the resulting forms together with the slice equations.

On $U_{\bar k}$ the support of this projective intersection is exactly $W_I^{(m)}(F)_{\mathrm{red}}$. By \cref{lem:principal-open-closure}, passing to the principal open and taking projective closure deletes only components contained in the boundary and cannot increase geometric degree. Repeated use of \cref{lem:mixed-bezout}, with every zero equation assigned the prescribed positive degree just specified, therefore gives
\[
 \gdeg\bigl(W_I^{(m)}(F)\bigr)
 \le\delta e^{\abs I}(d+\tau_m)^{N_m(r-\abs I)}.
\]
Since $e\le d$ and
$\abs I+N_m(r-\abs I)\le N_m(r)$,
the right side is bounded by $C_m(d+1)^{N_m}$ after enlarging $C_m$. The displayed combinatorial inequality follows by decreasing the argument of $N_m(s)$ one step at a time and using
$N_m(s)-N_m(s-1)=\binom{s+m-1}{m-1}\ge1$.
For a zero-dimensional reduced scheme over $\bar k$, every geometric point is a degree-one irreducible component, so \eqref{eq:number-points-zero} follows.
\end{proof}

\subsubsection{Extracting a controlled curve}

\begin{lemma}[Controlled curve extraction]\label{lem:curve-extraction}
If $Z\subseteq U_{\bar k}$ is closed and $\dim Z\ge1$, then there exists an integral curve
$C\subseteq Z_{\mathrm{red}}$
closed in $U_{\bar k}$ such that
\begin{equation}\label{eq:curve-degree-control}
 \deg\overline C\le\gdeg(Z).
\end{equation}
\end{lemma}

\begin{proof}
Choose a positive-dimensional irreducible component $Z_0$ of $Z_{\mathrm{red}}$ and let $\overline Z_0$ be its projective closure. More generally, if $Z_i$ is any irreducible component of $Z_{\mathrm{red}}$, then $Z_i$ is closed in $U_{\bar k}$ and therefore
$\overline Z_i\cap U_{\bar k}=Z_i$.
Consequently, the closures of distinct irreducible components are distinct and incomparable: an inclusion $\overline Z_i\subseteq\overline Z_j$ would give $Z_i\subseteq Z_j$ after intersection with $U_{\bar k}$. They are therefore precisely the irreducible components of the projective closure of $Z_{\mathrm{red}}$. In particular,
\begin{equation}\label{eq:component-closure-degree}
 \deg\overline Z_0\le\gdeg(Z).
\end{equation}
Since $Z_0$ is closed in $U_{\bar k}$, one has $Z_0=\overline Z_0\cap U_{\bar k}$. If $\dim Z_0=1$, take $C=Z_0$. Assume $\ell=\dim Z_0\ge2$, and let
\[
 E=\overline Z_0\setminus Z_0
   =\overline Z_0\cap(\overline U_{\bar k}\setminus U_{\bar k}).
\]
Then $E$ is closed and $\dim E\le\ell-1$.

Set $Z^{(0)}=\overline Z_0$ and $E^{(0)}=E_{\mathrm{red}}$. Inductively, after $H_1,\ldots,H_{h-1}$ have been chosen, choose a hyperplane $H_h$ that contains no positive-dimensional irreducible component of either $Z^{(h-1)}$ or $E^{(h-1)}$, and put
\[
 Z^{(h)}=\bigl(Z^{(h-1)}\cap H_h\bigr)_{\mathrm{red}},
 \qquad
 E^{(h)}=\bigl(E^{(h-1)}\cap H_h\bigr)_{\mathrm{red}}.
\]
There are only finitely many components to avoid, and the hyperplanes containing any one of them form a proper closed subset of the dual projective space. Since $\bar k$ is infinite, such an $H_h$ exists. For every positive-dimensional projective irreducible component, a hyperplane not containing it has nonempty intersection of dimension exactly one less. It follows inductively that
$\dim Z^{(h)}=\ell-h, \qquad \dim E^{(h)}\le\max\{\ell-1-h,-\infty\}$.
After $\ell-1$ steps, the reduced scheme
\[
 L=Z^{(\ell-1)}
 =\bigl(\overline Z_0\cap H_1\cap\cdots\cap H_{\ell-1}\bigr)_{\mathrm{red}}
\]
has dimension one, while $(L\cap E)_{\mathrm{red}}\subseteq E^{(\ell-1)}$ is zero-dimensional or empty. Hence no one-dimensional irreducible component of $L$ is contained in $E$. Choose one such component $\overline C$ and put $C=\overline C\cap U$. Then $C$ is an integral curve, closed in $U_{\bar k}$, and contained in $Z_0$.

Applying \cref{lem:mixed-bezout} successively to these degree-one sections gives
$\gdeg(L)\le\deg\overline Z_0$.
Therefore, by \eqref{eq:component-closure-degree},
$\deg\overline C\le\deg\overline Z_0\le\gdeg(Z)$.
\end{proof}

Because $t:U\to\A^r_k$ is \'{e}tale, it is quasi-finite. This forces a coordinate to vary on every curve.

\begin{lemma}[A varying unsliced coordinate]\label{lem:varying-coordinate}
Let $I\subsetneq[r]$, let $C\subseteq Y_{I,\bar k}$ be an integral curve, and assume each $Q_i$ is irreducible. Then there exists
$j\in[r]\setminus I$
such that $t_j|_C$ is nonconstant.
\end{lemma}

\begin{proof}
For $i\in I$, the regular function $t_i$ on the integral $\bar k$-curve $C$ satisfies $Q_i(t_i)=0$. Since $Q_i$ splits into distinct linear factors over $\bar k$, the image of $C$ under $t_i$ is contained in a finite set; irreducibility implies that $t_i|_C$ is constant. If every $t_j$ for $j\notin I$ were also constant, then $t|_C$ would be constant. This would place $C$ in a fiber of the quasi-finite morphism $t$, which is impossible.
\end{proof}

\subsubsection{How many values can a curve omit?}

Recall from \eqref{eq:boundary-principal} that
$U=\overline U\cap D_+(B)$, with $\deg B\le b$.

\begin{lemma}[Boundary-value lemma]\label{lem:boundary-values}
Let $C\subseteq U_{\bar k}$ be an integral curve and suppose $t_j|_C$ is nonconstant. Then
\begin{equation}\label{eq:missing-value-bound}
 \#\left(\bar k\setminus t_j(C(\bar k))\right)
 \le b\deg\overline C.
\end{equation}
Consequently, if
$C\subseteq W_I^{(m)}(F)$ and $\deg\overline C\le\Gamma_m(d)$,
then
$\#\left(\bar k\setminus t_j(C(\bar k))\right)\le b\Gamma_m(d)$.
\end{lemma}

\begin{proof}
Because $C$ is closed in $U_{\bar k}$, one has $C=\overline C\cap U_{\bar k}$. Let $\nu:\widetilde C\to\overline C$ be the normalization. The nonconstant rational function $t_j$ extends to a nonconstant morphism
$\widetilde t_j:\widetilde C\longrightarrow\Pj^1_{\bar k}$,
which is finite and surjective. The boundary
$S_C=\widetilde C\setminus\nu^{-1}(C)$
is finite. If $a\in\bar k$ is omitted by $t_j(C)$, then the nonempty fiber $\widetilde t_j^{-1}(a)$ is contained in $S_C$. Fibers over distinct values are disjoint, so
$\#\bigl(\bar k\setminus t_j(C(\bar k))\bigr)\le\#S_C$.

Every point of $S_C$ lies over $\overline C\cap V(B)$. Put $b_B=\deg B\le b$. Since $C\subseteq D_+(B)$, the curve $\overline C$ is not contained in $V(B)$, so the pullback of $B$ is a nonzero section of
$\nu^*\mathcal O_{\overline C}(b_B)$. Its zero divisor has degree
$\deg\operatorname{div}(\nu^*B) =b_B\deg\overline C \le b\deg\overline C$.
Every point of $S_C$ occurs in this divisor with positive multiplicity. Hence $\#S_C\le b\deg\overline C$, proving \eqref{eq:missing-value-bound}; the final assertion is immediate.
\end{proof}

\subsubsection{Irreducible Frobenius-orbit slices}\label{sec:orbits}

Let $\cI_e(q)$ denote the set of monic irreducible polynomials of degree $e$ in $\F_q[T]$, and put
$I_e(q)=\#\cI_e(q)$.
The exact formula is
\begin{equation}\label{eq:number-irreducibles}
 I_e(q)=\frac1e\sum_{a\mid e}\mu(a)q^{e/a}.
\end{equation}

\begin{lemma}[A uniform lower bound]\label{lem:irreducible-count}
For every prime power $q\ge2$ and every $e\ge1$,
\begin{equation}\label{eq:irreducible-lower}
 I_e(q)\ge\frac{q^e}{2e}.
\end{equation}
\end{lemma}

\begin{proof}
For $e=1$, this is immediate. For $e\ge2$, \eqref{eq:number-irreducibles} gives
$eI_e(q) \ge q^e-\sum_{m=1}^{\floor{e/2}}q^m$.
For $e=2,3$ the desired inequality is checked directly. If $e\ge4$, then
$\sum_{m=1}^{\floor{e/2}}q^m \le 2q^{e/2} \le\frac{q^e}{2}$,
where the last inequality uses $q^{e/2}\ge4$. This proves \eqref{eq:irreducible-lower}. See also \cite[Chapter~3]{LidlNiederreiter1997}.
\end{proof}

An irreducible $Q\in\cI_e(q)$ determines a Frobenius orbit
$\orb(Q)=\set{\alpha,\alpha^q,\ldots,\alpha^{q^{e-1}}}\subseteq\overline{\F}_q$
consisting of its roots. Distinct monic irreducible polynomials have disjoint root orbits.

\begin{proposition}[Orbit-slice hitting]\label{prop:orbit-hitting}
Let $C\subseteq U_{\overline{\F}_q}$ be an integral curve, let $t_j|_C$ be nonconstant, and suppose
$\deg\overline C\le\Gamma$.
Choose $Q$ uniformly from $\cI_e(q)$. Then
\begin{equation}\label{eq:orbit-miss}
 \Prob_Q\left(C\cap V(Q(t_j))=\varnothing\right)
 \le 2b\Gamma q^{-e}.
\end{equation}
\end{proposition}

\begin{proof}
Let
$M_C=\overline{\F}_q\setminus t_j(C(\overline{\F}_q))$.
By \cref{lem:boundary-values}, $\#M_C\le b\Gamma$. If
$C\cap V(Q(t_j))=\varnothing$, then every root of $Q$ belongs to $M_C$, so
$\orb(Q)\subseteq M_C$.
The root orbits for distinct $Q$ are disjoint and each has cardinality $e$. Hence the number of missed polynomials is at most $\#M_C/e$. Notice that no descent hypothesis on $C$ is used here: the finite set $M_C$ need not be Frobenius-stable. Using \cref{lem:irreducible-count},
\[
 \Prob_Q(C\cap V(Q(t_j))=\varnothing)
 \le \frac{\#M_C/e}{q^e/(2e)}
 \le 2b\Gamma q^{-e}.
\]
\end{proof}

The estimate depends on the boundary only through $b$. This is the point at which quasiprojectivity is handled: a fixed divisor need not meet $C$, but a random closed orbit in the coordinate line misses $C$ only when all of its roots fall among the finitely many omitted coordinate values.

\subsubsection{The full normal \texorpdfstring{$Q$}{Q}-adic tower}\label{sec:normal}

Fix $m\ge1$, assume $\operatorname{char}k>m$, and let $1\le e\le d$. Fix
$I\subsetneq[r],\qquad j\notin I,\qquad Q=Q_j\in\cI_e(q)$.
For uniform $F\in\cR_I(d)$, write the independent expansion
\begin{equation}\label{eq:normal-qadic-expansion}
 F=\sum_{a=0}^{\floor{d/e}}Q(t_j)^aF_a,
 \qquad
 F_a\in\cR_{I\cup\set{j}}(d-ae),
\end{equation}
as in \cref{cor:uniform-randomness}.

\begin{lemma}[Residue-degree divisibility]\label{lem:degree-divisibility}
Let $y$ be a closed point of $Y_{I\cup\set{j}}$. Evaluation of $t_j$ induces an embedding
\begin{equation}\label{eq:subfield-embedding}
 \F_q[T]/(Q)\hookrightarrow\kappa(y).
\end{equation}
In particular,
$e\mid[\kappa(y):\F_q]$.
\end{lemma}

\begin{proof}
The element $t_j(y)\in\kappa(y)$ is a root of the irreducible polynomial $Q$. Its minimal polynomial over $\F_q$ is therefore $Q$, which proves \eqref{eq:subfield-embedding}; the divisibility follows from the tower law for finite fields.
\end{proof}

\begin{lemma}[Evaluation entropy]\label{lem:evaluation-entropy}
Let $y$ be as in \cref{lem:degree-divisibility}, and let $L\ge0$. The evaluation map
$\operatorname{ev}_y:\cR_{I\cup\set{j}}(L)\longrightarrow\kappa(y)$
has image dimension at least
\begin{equation}\label{eq:evaluation-rank}
 \dim_{\F_q}\im(\operatorname{ev}_y)
 \ge\min\set{e,L+1}.
\end{equation}
\end{lemma}

\begin{proof}
Put $h=\min\set{e,L+1}$. The polynomials
$1,t_j,\ldots,t_j^{h-1}$
belong to $\cR_{I\cup\set{j}}(L)$. Their values at $y$ are linearly independent over $\F_q$, because the minimal polynomial of $t_j(y)$ has degree $e$.
\end{proof}

\begin{lemma}[Normal-tower entropy]\label{lem:normal-tower}
Condition on $F_0$, and let $y$ be a closed point of
$W_{I\cup\set{j}}^{(m)}(F_0)$. Then
\begin{equation}\label{eq:tower-probability}
 \Prob\left(
 \partial_j^{[a]}F(y)=0\text{ for }1\le a\le m
 \,\middle|\,F_0
 \right)
 \le q^{-\Lambda_m(d,e)},
\end{equation}
where
$\Lambda_m(d,e)=\min\set{me,d-e+1}$.
\end{lemma}

\begin{proof}
By \cref{lem:etale-taylor}, setting $\xi_j=u$ and $\xi_\ell=0$ for $\ell\ne j$ gives the truncated Taylor homomorphism
\[
 \Theta_j(G)=\sum_{a=0}^m\partial_j^{[a]}G\,u^a
 \quad\bmod u^{m+1},
\]
and it satisfies
$\Theta_j\bigl(Q(t_j)\bigr)=Q(t_j+u)$.
Modulo $Q(t_j)$, the constant term of $Q(t_j+u)$ is zero and its linear term is $Q'(t_j)u$. Consequently, for $1\le a\le m$,
\begin{equation}\label{eq:triangular}
 \partial_j^{[a]}F
 \equiv Q'(t_j)^aF_a+\Phi_a(F_0,\ldots,F_{a-1})
 \pmod{Q(t_j)},
\end{equation}
where $F_a$ is interpreted as zero if $d-ae<0$. Indeed, the coefficient of $u^a$ in
$Q(t_j+u)^a\Theta_j(F_a)$ is $Q'(t_j)^aF_a$, powers $Q^b$ with $b>a$ cannot contribute, and all terms with $b<a$ depend only on earlier blocks. Since $Q$ is separable, $Q'(t_j)$ is a unit on $Y_{I\cup\set{j}}$.

For $a\ge0$, let
$\mathscr F_a=\sigma(F_0,\ldots,F_a)$,
and, for $0\le a\le m$, let
$\mathcal A_a=\set{\partial_j^{[b]}F(y)=0\text{ for }1\le b\le a}$, and set $\mathcal A_0=\Omega$.
Fix a realization of $F_0$. The point $y$, its residue field, and every constant term in the equations below are then fixed. The event $\mathcal A_{a-1}$ is $\mathscr F_{a-1}$-measurable. For every realization of $\mathscr F_{a-1}$ lying in $\mathcal A_{a-1}$, the right side of \eqref{eq:triangular} is an affine function of the still independent and uniform block $F_a$. Its linear part is evaluation at $y$, followed by multiplication by the unit $Q'(t_j)^a$. By \cref{lem:evaluation-entropy}, whenever
$\Prob(\mathcal A_{a-1}\mid F_0)>0$,
\[
 \Prob\left(
 \mathcal A_a
 \,\middle|\,
 \mathcal A_{a-1},F_0
 \right)
 \le q^{-\rho_a},
 \qquad
 \rho_a=\max\set{0,\min\set{e,d-ae+1}}.
\]
If $d-ae<0$, then $F_a=0$ and $\rho_a=0$, so this inequality remains valid. Since $\mathcal A_{a-1}$ depends only on $F_0,\ldots,F_{a-1}$, conditioning on it does not alter the uniform distribution of the independent block $F_a$. We now prove inductively that
\begin{equation}\label{eq:normal-tower-induction}
 \Prob(\mathcal A_a\mid F_0)
 \le q^{-\sum_{b=1}^a\rho_b}
 \qquad(0\le a\le m).
\end{equation}
The assertion is immediate for $a=0$. If
$\Prob(\mathcal A_{a-1}\mid F_0)=0$, it is also immediate for $a$. Otherwise,
\[
 \Prob(\mathcal A_a\mid F_0)
 =\Prob(\mathcal A_{a-1}\mid F_0)
  \Prob(\mathcal A_a\mid\mathcal A_{a-1},F_0)
 \le q^{-\sum_{b=1}^a\rho_b}.
\]
Taking $a=m$ in \eqref{eq:normal-tower-induction} gives the required conditional bound.
Writing $L=d+1$, the number $\rho_a$ is the real length of the half-open interval
$[ae,(a+1)e)\cap[0,L)$.
For $1\le a\le m$, these half-open intervals are pairwise disjoint and consecutive, and their union is
$[e,(m+1)e)\cap[0,L)$.
Because $e\le d$, one has $L>e$. Therefore
\[
 \sum_{a=1}^m\rho_a
 =\operatorname{length}_{\mathbf R}\bigl([e,(m+1)e)\cap[0,L)\bigr)
 =\min\set{me,L-e}
 =\min\set{me,d-e+1}.
\]
This proves the claim.
\end{proof}

\begin{remark}\label{rem:first-order-block}
For $m=1$, the full expansion reduces at the relevant level to
$F=F_0+Q(t_j)H$,
and \eqref{eq:triangular} becomes the familiar affine-linear normal derivative
$\partial_jF \equiv\partial_jF_0+Q'(t_j)H \pmod{Q(t_j)}$.
Thus the detailed one-block entropy estimate is exactly the first stage of the normal tower.
\end{remark}

\section{Recursive estimates and arithmetic globalization}\label{sec:recursion}

Fix $m,d,e,q$ with
$m\ge1,\qquad \operatorname{char}\F_q>m,\qquad 1\le e\le d$.
Put
$\varepsilon_m(d,e,q)=2b\Gamma_m(d)q^{-e}$.
For $0\le s\le r$, let $\beta_s(d,e,q)$ be the supremum of
$\Prob_{F\in\cR_I(d)}\bigl(\dim W_I^{(m)}(F)\ge1\bigr)$
over all subsets $I\subseteq[r]$ with $\abs I=s$ and all previously chosen tuples of irreducible degree-$e$ orbit polynomials. By \eqref{eq:terminal-dimension},
$\beta_r(d,e,q)=0$.

\begin{proposition}[One recursive step]\label{prop:jet-recursion}
If $0\le s<r$ and $\varepsilon_m(d,e,q)<1$, then
\begin{equation}\label{eq:jet-recursion}
 \beta_s(d,e,q)
 \le
 \frac{r-s}{1-\varepsilon_m(d,e,q)}
 \left(
 \beta_{s+1}(d,e,q)
 +\Gamma_m(d)q^{-\Lambda_m(d,e)}
 \right).
\end{equation}
\end{proposition}

\begin{proof}
Fix $I$, $\abs I=s$, and the prior orbit polynomials. Let
$E_I(F)=\set{\dim W_I^{(m)}(F)\ge1}$.
For every $F$ satisfying $E_I(F)$, choose an integral curve
$C_F\subseteq W_I^{(m)}(F)_{\overline{\F}_q}$ with $\deg\overline C_F\le\Gamma_m(d)$,
using \cref{lem:degree-WI-m,lem:curve-extraction}. Since $C_F\subseteq Y_{I,\overline{\F}_q}$, every sliced coordinate is constant on $C_F$. By \cref{lem:varying-coordinate}, there is an index
$j_F\in[r]\setminus I$
such that $t_{j_F}|_{C_F}$ is nonconstant. The probability space for $F$ is finite, so these choices may be fixed before introducing the auxiliary randomness.

Choose an index $J$ uniformly from $[r]\setminus I$ and, independently, choose
$Q\in\cI_e(q)$ uniformly. Conditional on $J=j_F$, \cref{prop:orbit-hitting} gives
$\Prob_Q\bigl(C_F\cap V(Q(t_J))\ne\varnothing\bigr) \ge1-\varepsilon_m(d,e,q)$.
Therefore
\begin{equation}\label{eq:bad-to-hit}
 \frac{1-\varepsilon_m(d,e,q)}{r-s}\one_{E_I(F)}
 \le
 \Prob_{J,Q}\left(
 W_I^{(m)}(F)\cap V(Q(t_J))\ne\varnothing
 \right).
\end{equation}
Averaging over uniform $F\in\cR_I(d)$ gives
\begin{equation}\label{eq:averaged-bad-to-hit}
 \frac{1-\varepsilon_m(d,e,q)}{r-s}\Prob_F(E_I(F))
 \le
 \Prob_{F,J,Q}\left(
 W_I^{(m)}(F)\cap V(Q(t_J))\ne\varnothing
 \right).
\end{equation}

Fix now $j\notin I$ and $Q\in\cI_e(q)$, and use the expansion
$F=\sum_{a\ge0}Q(t_j)^aF_a$.
On the new slice $Y_{I\cup\set{j}}$, every jet equation involving only directions different from $j$ is the corresponding equation for $F_0$. Hence, on reduced supports,
\begin{equation}\label{eq:jet-intersection-inclusion}
 W_I^{(m)}(F)\cap V(Q(t_j))
 \subseteq
 W_{I\cup\set{j}}^{(m)}(F_0)
 \cap\bigcap_{a=1}^mV\bigl(\partial_j^{[a]}F\bigr).
\end{equation}
If $W_{I\cup\set{j}}^{(m)}(F_0)$ is positive-dimensional, the event is charged to $\beta_{s+1}$. Otherwise, condition on $F_0$. At this point the finite set of closed points of $W_{I\cup\set{j}}^{(m)}(F_0)$, their residue fields, and the affine constant terms in all normal equations are fixed; only the independent blocks $F_1,F_2,\ldots$ remain random. The reduced support has at most $\Gamma_m(d)$ geometric points and hence at most that many closed points. All schemes and functions on the right side of \eqref{eq:jet-intersection-inclusion} are defined over $\F_q$. Hence, if that right side contains a geometric point, its finite Frobenius orbit is contained in the same intersection and determines a closed point $y$ of $W_{I\cup\set{j}}^{(m)}(F_0)$. Each pure normal equation then vanishes as an element of $\kappa(y)$. At each such closed point, \cref{lem:normal-tower} bounds the conditional probability of those equations by
$q^{-\Lambda_m(d,e)}$. The conditional union bound, followed by averaging over $F_0$, therefore gives
\begin{align}
 &\Prob_F\left(
 W_I^{(m)}(F)\cap V(Q(t_j))\ne\varnothing
 \right)\notag\\
 &\qquad\le
 \beta_{s+1}(d,e,q)
 +\Gamma_m(d)q^{-\Lambda_m(d,e)}.
 \label{eq:fixed-slice-upper}
\end{align}
No independence between distinct point conditions is used. Averaging \eqref{eq:fixed-slice-upper} over $J$ and $Q$, and combining with \eqref{eq:averaged-bad-to-hit}, proves \eqref{eq:jet-recursion}.
\end{proof}

\begin{theorem}[Local higher-jet orbit-slicing estimate]\label{thm:local-jet}
Assume
$\varepsilon_m(d,e,q)\le\frac12$.
Then
\begin{equation}\label{eq:local-jet}
 \Prob_{F\in A_{\le d}}
 \bigl(\dim W_\varnothing^{(m)}(F)\ge1\bigr)
 \le A_r\Gamma_m(d)q^{-\Lambda_m(d,e)},
\end{equation}
where
\begin{equation}\label{eq:Ar}
 A_0=0,
 \qquad
 A_r=2^rr!\sum_{h=0}^{r-1}\frac1{2^hh!}
 <\exp(1/2)2^rr!\quad(r\ge1).
\end{equation}
\end{theorem}

\begin{proof}
If $r=0$, then $I=[r]=\varnothing$ in \eqref{eq:terminal-dimension}, so the event in \eqref{eq:local-jet} is empty and the asserted bound follows from $A_0=0$. Assume $r\ge1$. Since $\beta_r=0$, \cref{prop:jet-recursion} gives
$\beta_s\le2(r-s)(\beta_{s+1}+a)$, where $a=\Gamma_m(d)q^{-\Lambda_m(d,e)}$.
The recurrence
$c_0=0,\qquad c_\ell=2\ell(c_{\ell-1}+1)$
has the solution
$c_\ell=2^\ell\ell!\sum_{h=0}^{\ell-1}\frac1{2^hh!}$.
Backward induction gives $\beta_0\le A_ra$, proving \eqref{eq:local-jet}. The strict upper bound in \eqref{eq:Ar} follows by comparison with the exponential series.
\end{proof}

\subsection{The optimal orbit degree}

\begin{lemma}[Optimal balance]\label{lem:optimal-e}
For every $d\ge1$ and
$e_m(d)=\ceil{\frac{d+1}{m+1}}$,
one has
\begin{equation}\label{eq:optimal-Lambda}
 \Lambda_m\bigl(d,e_m(d)\bigr)
 =\lambda_m(d)
 =\floor{\frac{m(d+1)}{m+1}}.
\end{equation}
No integer $1\le e\le d$ gives a larger value of $\Lambda_m(d,e)$.
\end{lemma}

\begin{proof}
For every $1\le e\le d$,
$\Lambda_m(d,e)=\min\set{me,d-e+1}$.
If $(m+1)e\le d+1$, then $me\le m(d+1)/(m+1)$. If $(m+1)e\ge d+1$, then
$d-e+1\le m(d+1)/(m+1)$. For $e=e_m(d)$ one has
$(m+1)e\ge d+1$, hence $me\ge d-e+1$, and
\[
 \Lambda_m(d,e)
 =d-e+1
 =d+1-\ceil{\frac{d+1}{m+1}}
 =\floor{\frac{m(d+1)}{m+1}}.
\]
Since $\Lambda_m(d,e)$ is integral, the real upper bound proves optimality among integer $e$.
\end{proof}

\begin{corollary}[Balanced local estimate]\label{cor:balanced-local-jet}
For the fixed chart $U/\F_q$ and every fixed $m\ge1$ with
$\operatorname{char}\F_q>m$, there is $d_{U,m}$ such that, for every $d\ge d_{U,m}$,
\begin{equation}\label{eq:balanced-local-jet}
 \Prob_{F\in A_{\le d}}
 \bigl(\dim W_\varnothing^{(m)}(F)\ge1\bigr)
 \le A_r\Gamma_m(d)q^{-\lambda_m(d)}.
\end{equation}
If $U$ ranges over the finite-field fibers of one controlled integral chart from
\cref{prop:controlled-charts}, then $C_m$, $\Gamma_m$, and the threshold
$d_{U,m}$ may be chosen uniformly over all fibers of characteristic greater than $m$.
\end{corollary}

\begin{proof}
Take $e=e_m(d)$. Since $e$ grows linearly with $d$ while $\Gamma_m(d)$ is polynomial, there is $d_{U,m}$ such that
$2b\Gamma_m(d)2^{-e}\le\frac12$
for all $d\ge d_{U,m}$. Then \cref{thm:local-jet,lem:optimal-e} apply because $q\ge2$.
For the fibers of a fixed controlled integral chart, the constants $b$, $\tau_m$, and the geometric-degree bound entering $C_m$ are uniform by
\cref{prop:controlled-charts,lem:higher-representatives,lem:degree-WI-m}; hence the same choice of $C_m$ and $d_{U,m}$ works for every such fiber.
\end{proof}

\subsection{Globalization over the arithmetic family}\label{sec:global}

Let the controlled cover and constants be as in \cref{prop:controlled-charts}. For a prime $p\nmid N_0$, write
$U_{\alpha,p}=\cU_\alpha\times\F_p$.
These opens cover $\cX_p$.

\subsubsection{Homogeneous forms and affine polynomials}

On a standard projective chart $x_a\ne0$, dehomogenization is a vector-space isomorphism
\begin{equation}\label{eq:dehomogenization}
 S_{d,p}
 \xrightarrow{\sim}
 \F_p[z_1,\ldots,z_n]_{\le d},
 \qquad
 f\longmapsto x_a^{-d}f.
\end{equation}
Thus a uniform homogeneous form becomes a uniform affine polynomial of total degree at most $d$.

\begin{lemma}[Detection on one chart]\label{lem:detection-chart}
If
$\dim J_m(f)\ge1$,
then there exists a controlled chart $U_{\alpha,p}$ such that
$\dim Z\bigl(j^m(f|_{U_{\alpha,p}})\bigr)\ge1$.
\end{lemma}

\begin{proof}
Base change the finite controlled cover to $\overline{\F}_p$. Choose a positive-dimensional irreducible component
$Z\subseteq\bigl(J_m(f)_{\overline{\F}_p}\bigr)_{\mathrm{red}}$
and let $\eta_Z$ be its generic point. Some base-changed chart
$U_{\alpha,p}\times_{\F_p}\overline{\F}_p$ contains $\eta_Z$. Vanishing of the principal-parts jet is local and commutes with this field extension. Hence the nonempty open subset of $Z$ cut out by that chart lies in the corresponding local jet-zero scheme and has dimension $\dim Z\ge1$. Since Krull dimension is invariant under extension of the ground field, the local jet-zero scheme over $\F_p$ itself has dimension at least one.
\end{proof}

\begin{proof}[Proof of \cref{thm:jet-global}]
If $r=0$, every fiber $\cX_p$ is zero-dimensional and the event is empty. Assume $r\ge1$. Choose $p_0$ larger than $m$ and every prime divisor of $N_0$. By the uniformity assertion in \cref{cor:balanced-local-jet}, the finitely many controlled integral charts admit common constants and a common threshold $d_0$, depending only on the fixed embedding $\cX\subseteq\Pj^n_{\Z}$ and on $m$.

For a fixed chart, \eqref{eq:dehomogenization} identifies the uniform distribution on $S_{d,p}$ with the uniform distribution on the corresponding affine polynomial space. By \cref{lem:WI-interpretation}, the local zero scheme is $W_\varnothing^{(m)}(F)$. Hence \cref{lem:detection-chart} and a union bound over the finite cover give
\[
 \Prob_{f\in S_{d,p}}\bigl(\dim J_m(f)\ge1\bigr)
 \le
 M A_r C_m(d+1)^{N_m}p^{-\lambda_m(d)}.
\]
Absorbing the fixed factor $MA_rC_m$ into $C$ proves \eqref{eq:jet-global}.
\end{proof}

\begin{proof}[Proof of \cref{cor:poonen}]
For $m=1$, the local and global first jet-zero schemes have the same support as the Bertini singular locus, and
$N_1=r+1$ and $\lambda_1(d)=\floor{\frac{d+1}{2}}=\ceil{\frac d2}$.
Poonen's notation $(H_f\cap\cX_p)_{\mathrm{sing}}$ denotes precisely the closed set where the section is not smooth of the expected dimension $r-1$, so its support is $\BSing(H_f\cap\cX_p)$. Thus \eqref{eq:first-order-global} is the case $m=1$ of \cref{thm:jet-global}.

Fix $A>0$. Choose the degree threshold so large that
$\ceil{\frac d2}\ge A$ and $C(d+1)^{r+1}\le2^{\ceil{d/2}-A}$.
Since $p\ge2$,
\[
 C(d+1)^{r+1}p^{-\ceil{d/2}}
 \le p^{\ceil{d/2}-A}p^{-\ceil{d/2}}
 =p^{-A}.
\]
Taking $A=2$ gives a bound at most $p^{-2}$; hence \eqref{eq:Poonen-conjecture} follows, for example with any fixed $c>1$.
\end{proof}

\section{Complete intersections, consequences, and context}\label{sec:ci}

\begin{lemma}[Jacobian criterion for section complete intersections]\label{lem:ci-jacobian}
Let $X$ be smooth of pure dimension $r$ over a field $k$, let $1\le c\le r$, let $L_1,\ldots,L_c$ be invertible sheaves on $X$, and let $s_i\in H^0(X,L_i)$. Put
$Y=Z(s_1,\ldots,s_c)$.
The local rule
\begin{equation}\label{eq:ci-conormal-map-general}
 \Phi_{\mathbf s}:\bigoplus_{i=1}^cL_i^{-1}\restr Y
 \longrightarrow\Omega^1_{X/k}\restr Y,
 \qquad
 e_i^\vee\longmapsto dG_i\restr Y
\end{equation}
for $s_i=G_i e_i$ defines a canonical homomorphism. The support of
$D_{c-1}(\Phi_{\mathbf s})$ is exactly the locus where $Y$ is not smooth of relative dimension $r-c$ over $k$. Consequently,
$D_{c-1}(\Phi_{\mathbf s})=\varnothing$
if and only if $Y$ is empty or smooth of pure dimension $r-c$.
\end{lemma}

\begin{proof}
If $e_i'=u_i e_i$, then $G_i'=u_i^{-1}G_i$ and
$dG_i'=u_i^{-1}dG_i-u_i^{-2}G_i\,du_i$.
After restriction to $Y$ the second term vanishes, so the local maps in \eqref{eq:ci-conormal-map-general} glue according to the transition functions of $L_i^{-1}$.

It remains to identify the rank locus. Rank and smoothness are preserved under extension of the ground field, and every nonempty finite-type scheme over a field has a geometric point, so it is enough to compare the two loci at closed points after extending $k$ to an algebraic closure. Let $x\in Y$ be such a closed geometric point. Trivialize the $L_i$ near $x$ and write $s_i=G_i e_i$. If
$\operatorname{rank}\bigl(dG_1(x),\ldots,dG_c(x)\bigr)=c$,
then the Jacobian criterion shows that the common zero scheme is smooth of codimension $c$, hence of dimension $r-c$, at $x$.

Conversely, suppose that $Y$ is smooth of dimension $r-c$ at $x$. Let
$R=\mathcal O_{X,x}$, $I=(G_1,\ldots,G_c)$, and $S=R/I=\mathcal O_{Y,x}$.
The rings $R$ and $S$ are regular local of dimensions $r$ and $r-c$, respectively. By the regular-local quotient theorem \cite[Lemma 10.106.4, Tag 00NN]{StacksProject}, there is a regular system of parameters
$x_1,\ldots,x_r$ of $R$ such that
$I=(x_1,\ldots,x_c)$.
Thus $I$ is generated by a regular sequence and $I/I^2$ is a free $S$-module of rank $c$, with basis the classes of $x_1,\ldots,x_c$. The given elements $G_1,\ldots,G_c$ generate $I$, so their classes induce a surjection
$S^c\longrightarrow I/I^2$.
This is a surjection between free modules of the same finite rank over the local ring $S$, and hence is an isomorphism. Therefore the classes of the $G_i$ form an $S$-basis of $I/I^2$.

Write
$G_i=\sum_{j=1}^c a_{ij}x_j \qquad(1\le i\le c)$
with $a_{ij}\in R$. Modulo $I$, the matrix $\overline A=(\overline a_{ij})$ is the change-of-basis matrix from the classes of $x_1,\ldots,x_c$ to the classes of $G_1,\ldots,G_c$ in $I/I^2$; hence $\overline A\in\operatorname{GL}_c(S)$. Its image $A(x)$ in $\operatorname{Mat}_c(\kappa(x))$ is therefore invertible. Differentiating and evaluating at $x$ gives
$dG_i(x)=\sum_{j=1}^c a_{ij}(x)\,dx_j(x)$,
because every $x_j$ vanishes at $x$. Since $X$ is smooth at the geometric point $x$, the covectors $dx_1(x),\ldots,dx_r(x)$ form a basis of the cotangent space. The invertibility of $A(x)$ therefore implies that $dG_1(x),\ldots,dG_c(x)$ are linearly independent. This proves the asserted equality of supports and the final equivalence.
\end{proof}

We now prove \cref{thm:ci-global}. Fix one controlled chart $U$ and polynomials
$G_1,\ldots,G_h\in A$, with $0\le h<r$,
of degrees at most $D$. For
$K=\set{k_1,\ldots,k_h}\subseteq[r]$
put
$M_K=(D_{k_b}G_a)_{1\le a,b\le h}$ and $\Delta_K=\det M_K$,
with $\Delta_\varnothing=1$, and define
$U_K=U\cap V(G_1,\ldots,G_h)\cap D(\Delta_K)$.
Let $J=[r]\setminus K$. For $j\in J$, write
$v_j=(D_jG_a)_{a=1}^h$
and set
\begin{equation}\label{eq:pivot-derivation}
 E_{K,j}=\Delta_KD_j-
 \sum_{b=1}^h\bigl(\operatorname{adj}(M_K)v_j\bigr)_bD_{k_b}.
\end{equation}
For $h=0$, the sum is empty and $E_{\varnothing,j}=D_j$.

\begin{proposition}[Uniform Jacobian-pivot charts]\label{prop:pivot}
Every nonempty $U_K$ is smooth of pure dimension $r-h$, and
\begin{equation}\label{eq:pivot-etale-map}
 (t_j)_{j\in J}:U_K\longrightarrow\A^{r-h}_k
\end{equation}
is \'{e}tale. The derivations $E_{K,j}$ are tangent to $U_K$ and satisfy
\begin{equation}\label{eq:pivot-duality}
 E_{K,j}(t_\ell)=\Delta_K\Delta_j\delta_{j\ell}
 \qquad(j,\ell\in J).
\end{equation}
Let $\overline{U_K}$ denote the scheme-theoretic closure of $U_K$ in the fixed projective closure $\overline U$. There is a constant $C_h$, depending only on the original controlled chart and on $h$, such that
\begin{align}
 \gdeg(\overline{U_K})&\le C_h(D+1)^h,
 \label{eq:pivot-degree}\\
 \deg B_K&\le C_h(D+1),
 \label{eq:pivot-boundary}\\
 \deg\bigl(\text{chosen polynomial coefficient representatives of }E_{K,j}\bigr)&\le C_h(D+1),
 \label{eq:pivot-coefficients}
\end{align}
where $B_K$ is a homogeneous polynomial for which
\begin{equation}\label{eq:pivot-principal-open-scheme}
 U_K=\overline{U_K}\cap D_+(B_K)
\end{equation}
as schemes.
\end{proposition}

\begin{proof}
If $U_K=\varnothing$, then its scheme-theoretic closure is empty. Take $B_K=B$. The principal-open identity and the degree and boundary bounds are immediate, while the coefficient estimate follows from the final degree computation below, which does not use nonemptiness. Hence assume from now on that $U_K\ne\varnothing$.

Since $D_i=\Delta_i\partial_i$ and every $\Delta_i$ is a unit, one has
$dG_a=\sum_{\ell=1}^r\Delta_\ell^{-1}D_\ell(G_a)\,dt_\ell$.
Thus the determinant of the coefficient matrix in the pivot directions $K$ is
$\Delta_K/\prod_{b=1}^h\Delta_{k_b}$, a unit multiple of $\Delta_K$. On $D(\Delta_K)$ the differentials $dG_1,\ldots,dG_h$ are independent. The Jacobian criterion therefore shows that $U_K$ is smooth of codimension $h$ in $U$, hence of pure dimension $r-h$. Solving the pivot differentials in terms of the remaining differentials shows that the classes of $dt_j$, $j\in J$, form a basis of $\Omega^1_{U_K/k}$. Since both $U_K$ and $\A^{r-h}_k$ are smooth over $k$, the transitivity triangle for cotangent complexes gives
$L_{U_K/\A^{r-h}_k}=0$.
The morphism \eqref{eq:pivot-etale-map} is locally of finite presentation, so the infinitesimal criterion implies that it is \'{e}tale.

The identity
$M_K\operatorname{adj}(M_K)=\Delta_KI_h$
gives
$E_{K,j}(G_a)=0 \qquad(1\le a\le h)$.
Thus $E_{K,j}$ descends to the quotient by $(G_1,\ldots,G_h)$. The resulting derivation extends uniquely to the localization at $\Delta_K$, and hence defines a derivation on $U_K$. If $\ell\in J$, then
$D_{k_b}(t_\ell)=0$ and $D_j(t_\ell)=\Delta_j\delta_{j\ell}$, which proves \eqref{eq:pivot-duality}.

We now record the closure scheme, rather than only its reduced support. Let $S=k[x_0,\ldots,x_n]$. Homogenize every $G_a$ to the common prescribed degree $D+1$ with respect to the standard-chart coordinate. This does not change its zero scheme on $U$, and it also covers the case $G_a=0$ without invoking an undefined polynomial degree. Let $I(T_K)\subseteq S$ be the homogeneous ideal obtained by adding these homogenizations to the ideal of $\overline U$. Choose a polynomial representative of $\Delta_K$ in the fixed affine coordinates and let $\widetilde\Delta_K$ be its homogenization. Because $U_K$ is nonempty, this representative is nonzero on $U$ and hence is not the zero polynomial. Put
\[
 B_K=B\widetilde\Delta_K,
 \qquad
 I_K=I(T_K):B_K^\infty,
 \qquad
 \overline{U_K}=\Proj(S/I_K).
\]
The standard contraction identity for localization gives
\begin{equation}\label{eq:pivot-saturation-contraction}
 I_K=\bigl(I(T_K)S_{B_K}\bigr)\cap S
     =I(T_K):B_K^\infty.
\end{equation}
Thus $\Proj(S/I_K)$ is the scheme-theoretic closure in $\Pj^n_k$ of $T_K\cap D_+(B_K)$. In particular,
\[
 \overline{U_K}\cap D_+(B_K)
 =T_K\cap D_+(B_K)
 =U\cap V(G_1,\ldots,G_h)\cap D(\Delta_K)
 =U_K
\]
as schemes. On reduced supports, \eqref{eq:pivot-saturation-contraction} deletes precisely the irreducible components of $(T_K)_{\mathrm{red}}$ contained in $V(B_K)$. Hence $(\overline{U_K})_{\mathrm{red}}$ is the union of the remaining irreducible components. Therefore repeated application of \cref{lem:mixed-bezout}, with each equation assigned the prescribed degree $D+1$, gives
$\gdeg(\overline{U_K}) \le\delta(D+1)^h \le C_h(D+1)^h$.

For the remaining degree estimates, use the fixed polynomial representatives of the original derivations. From \eqref{eq:derivative-degree},
$\deg(D_iG_a)\le D+\tau$.
Hence a polynomial representative of $\Delta_K$ has degree at most $h(D+\tau)$, and
$\deg B_K\le b+h(D+\tau)\le C_h(D+1)$.
Every coefficient in \eqref{eq:pivot-derivation} is a sum of products consisting of at most $h$ entries $D_iG_a$ and one coefficient of an original derivation. Its degree is therefore at most $C_h(D+1)$. This proves \eqref{eq:pivot-degree}--\eqref{eq:pivot-coefficients} and the scheme-theoretic identity \eqref{eq:pivot-principal-open-scheme}.
\end{proof}

\begin{corollary}[Moving-chart hypersurface estimate]\label{cor:moving-local}
For the fixed finite-field chart $U/\F_q$, there are constants $C,D_0$, depending only on its controlled data and on $r$, such that for every $D\ge D_0$, every choice of $G_1,\ldots,G_h$ of degrees at most $D$, every pivot $K$, and every uniform $F\in A_{\le D}$,
\begin{equation}\label{eq:moving-local}
 \Prob_F\bigl(\dim\BSing_{U_K}(F)\ge1\bigr)
 \le C(D+1)^{r+1}q^{-\ceil{D/2}}.
\end{equation}
If $U$ ranges over the finite-field fibers of one controlled integral chart from
\cref{prop:controlled-charts}, the same constants $C,D_0$ may be chosen uniformly over all such fibers.
\end{corollary}

\begin{proof}
If $U_K$ is empty, there is nothing to prove. Put
$s=r-h$ and, for $j\in J$, $u_{K,j}=\Delta_K\Delta_j\in\Gamma(U_K,\mathcal O_{U_K})^\times$.
By \eqref{eq:pivot-duality}, the derivations
$\widetilde\partial_{K,j}=u_{K,j}^{-1}E_{K,j}$
are dual to the \'{e}tale coordinates $(t_j)_{j\in J}$ on $U_K$. In particular,
\begin{equation}\label{eq:moving-initial-singular-scheme}
 \BSing_{U_K}(F)
 =\bigl(V_{U_K}(F,E_{K,j}F:j\in J)\bigr)_{\mathrm{red}}.
\end{equation}

We now reconstruct the entire first-order orbit recursion on this moving chart. Fix temporarily an integer $e$ with $1\le e\le D$. For $I\subseteq J$ and monic irreducible polynomials $Q_i\in\F_q[T]$ of degree $e$, define
\begin{align}
 Y_{K,I}
 &=U_K\cap\bigcap_{i\in I}V(Q_i(t_i)),
 \label{eq:moving-YKI}\\
 W_{K,I}(F)
 &=Y_{K,I}\cap
   V\bigl(F,E_{K,\ell}F:\ell\in J\setminus I\bigr),
 \qquad F\in\cR_I(D).
 \label{eq:moving-WKI}
\end{align}
The map $(t_j)_{j\in J}:U_K\to\A^s$ is \'{e}tale, and every $Q_i$ is separable. Hence $Y_{K,I}$ is smooth of pure dimension $s-\abs I$, with the unsliced $t_\ell$, $\ell\in J\setminus I$, as \'{e}tale coordinates. For $i\in I$ and $\ell\in J\setminus I$, \eqref{eq:pivot-duality} gives
$E_{K,\ell}(Q_i(t_i))=0$.
It follows from the product rule that the equations in \eqref{eq:moving-WKI} depend only on the class of $F$ modulo the prior slice ideal $(Q_i(t_i):i\in I)$. Moreover, after division by the units $u_{K,\ell}$, they are exactly the equation of $F|_{Y_{K,I}}$ and all of its first tangential coordinate derivatives. Thus $W_{K,I}(F)$ is the first-jet zero scheme of $F|_{Y_{K,I}}$. In particular, \eqref{eq:moving-initial-singular-scheme} identifies the reduced support of $W_{K,\varnothing}(F)$ with $\BSing_{U_K}(F)$, while
\begin{equation}\label{eq:moving-terminal}
 \dim W_{K,J}(F)\le0.
\end{equation}

Let
$\delta_K=\gdeg(\overline{U_K})$ and $b_K=\deg B_K$,
and let $\tau_K$ bound the degrees of the polynomial coefficients of all $E_{K,j}$. Set
\begin{equation}\label{eq:moving-L-Gamma}
 L_K(D)=D+\tau_K+1,
 \qquad
 \Gamma_K(D)=\delta_K L_K(D)^{s+1}.
\end{equation}
At a stage with $a=\abs I$ orbit slices, the equations defining $W_{K,I}(F)$ consist of $a$ equations of degree $e$ and $s-a+1$ equations of degree at most $L_K(D)$. The scheme-theoretic principal-open description in \eqref{eq:pivot-principal-open-scheme}, followed by \cref{lem:principal-open-closure,lem:mixed-bezout}, therefore gives
\begin{equation}\label{eq:moving-degree-bound}
 \gdeg(W_{K,I}(F))
 \le\delta_Ke^aL_K(D)^{s-a+1}
 \le\Gamma_K(D),
\end{equation}
because $e\le D\le L_K(D)$. Hence a zero-dimensional $W_{K,I}(F)$ has at most $\Gamma_K(D)$ geometric points. If it is positive-dimensional, \cref{lem:curve-extraction} supplies an integral curve of projective degree at most $\Gamma_K(D)$. Since the coordinate map on $Y_{K,I}$ is \'{e}tale, at least one unsliced coordinate is nonconstant on that curve. Applying the boundary-value and orbit-hitting arguments with the principal boundary $B_K$ shows that a random degree-$e$ irreducible orbit slice in that coordinate misses the curve with probability at most $\varepsilon_K(D,e,q)$, where $\varepsilon_K(D,e,q)=2b_K\Gamma_K(D)q^{-e}$.

For $0\le a\le s$, let $\beta_{K,a}(D,e,q)$ be the supremum of
$\Prob_{F\in\cR_I(D)}\bigl(\dim W_{K,I}(F)\ge1\bigr)$
over all $I\subseteq J$ with $\abs I=a$ and all previously chosen degree-$e$ irreducible orbit polynomials. By \eqref{eq:moving-terminal}, $\beta_{K,s}=0$. We claim that, whenever $a<s$ and $\varepsilon_K(D,e,q)<1$,
\begin{equation}\label{eq:moving-recursion}
 \beta_{K,a}(D,e,q)
 \le
 \frac{s-a}{1-\varepsilon_K(D,e,q)}
 \left(
   \beta_{K,a+1}(D,e,q)
   +\Gamma_K(D)q^{-\mu(D,e)}
 \right),
 \qquad
 \mu(D,e)=\min\set{e,D-e+1}.
\end{equation}

To prove the claim, fix $I$ and the prior orbit polynomials. For every $F$ for which $W_{K,I}(F)$ is positive-dimensional, choose a controlled curve in its geometric reduced support and one unsliced coordinate $t_j$ that is nonconstant on that curve. Choose independently a uniform index in $J\setminus I$ and a uniform $Q\in\cI_e(q)$. Exactly as in \eqref{eq:bad-to-hit}--\eqref{eq:averaged-bad-to-hit}, averaging the orbit-hitting estimate gives
\begin{align}
 &\frac{1-\varepsilon_K(D,e,q)}{s-a}
 \Prob_F\bigl(\dim W_{K,I}(F)\ge1\bigr)
 \notag\\
 &\qquad\le
 \Prob_{F,j,Q}\bigl(W_{K,I}(F)\cap V(Q(t_j))\ne\varnothing\bigr).
 \label{eq:moving-bad-to-hit}
\end{align}

Fix now $j\in J\setminus I$ and $Q\in\cI_e(q)$. By \cref{cor:uniform-randomness}, the current uniform remainder has the exact independent decomposition
\begin{equation}\label{eq:moving-decomposition}
 F=F_0+Q(t_j)H,
 \quad
 F_0\sim\operatorname{Unif}(\cR_{I\cup\set{j}}(D)),
 \quad
 H\sim\operatorname{Unif}(\cR_I(D-e)).
\end{equation}
On the new slice $Q(t_j)=0$ one has
\begin{align}
 F&=F_0,
 \notag\\
 E_{K,\ell}F&=E_{K,\ell}F_0
 &&(\ell\in J\setminus(I\cup\set{j})),
 \label{eq:moving-tangential}\\
 E_{K,j}F&=E_{K,j}F_0+u_{K,j}Q'(t_j)H.
 \label{eq:moving-normal}
\end{align}
The first two identities imply the inclusion of geometric supports
\begin{equation}\label{eq:moving-intersection-inclusion}
 W_{K,I}(F)\cap V(Q(t_j))
 \subseteq
 W_{K,I\cup\set{j}}(F_0)\cap V(E_{K,j}F).
\end{equation}
If $W_{K,I\cup\set{j}}(F_0)$ is positive-dimensional, this event is charged to $\beta_{K,a+1}$. Otherwise condition on $F_0$. The closed points of $W_{K,I\cup\set{j}}(F_0)$ and the affine terms $E_{K,j}F_0(y)$ are then fixed, while $H$ remains independent and uniform. The reduced support has at most $\Gamma_K(D)$ geometric points and hence at most that many closed points. All schemes and functions in \eqref{eq:moving-intersection-inclusion} are defined over $\F_q$. Therefore, if its right side contains a geometric point, the corresponding Frobenius orbit determines a closed point $y$ of $W_{K,I\cup\set{j}}(F_0)$ at which $E_{K,j}F$ vanishes in $\kappa(y)$.

Fix one such closed point $y$. Since $Q(t_j(y))=0$ and $Q$ is irreducible, the minimal polynomial of $t_j(y)$ over $\F_q$ is $Q$. Put
$h_0=\mu(D,e)=\min\set{e,D-e+1}$.
The polynomials
$1,t_j,\ldots,t_j^{h_0-1}$
belong to $\cR_I(D-e)$ and have linearly independent values in $\kappa(y)$. Therefore the evaluation map
$\cR_I(D-e)\longrightarrow\kappa(y)$
has $\F_q$-rank at least $h_0$. The coefficient $u_{K,j}Q'(t_j)$ in \eqref{eq:moving-normal} is a unit at $y$. Since $H$ is independent of $F_0$ and uniform in $\cR_I(D-e)$, the conditional probability that the affine equation $E_{K,j}F(y)=0$ holds is at most $q^{-h_0}$. A union bound over the closed points gives
\begin{align}
 &\Prob_F\bigl(W_{K,I}(F)\cap V(Q(t_j))\ne\varnothing\bigr)
 \notag\\
 &\qquad\le
 \beta_{K,a+1}(D,e,q)
 +\Gamma_K(D)q^{-\mu(D,e)}.
 \label{eq:moving-fixed-slice}
\end{align}
No independence between conditions at distinct closed points is used. Combining \eqref{eq:moving-bad-to-hit} and \eqref{eq:moving-fixed-slice} proves \eqref{eq:moving-recursion}.

If $\varepsilon_K(D,e,q)\le1/2$, the same backward induction as in \cref{thm:local-jet} gives
\begin{equation}\label{eq:moving-recursion-solved}
 \Prob_F\bigl(\dim\BSing_{U_K}(F)\ge1\bigr)
 \le A_s\Gamma_K(D)q^{-\mu(D,e)}.
\end{equation}
We now substitute the moving-chart complexity. By \cref{prop:pivot},
$\delta_K\le C(D+1)^h$, $b_K\le C(D+1)$, and $\tau_K\le C(D+1)$.
Since $s=r-h$, \eqref{eq:moving-L-Gamma} gives
\begin{equation}\label{eq:moving-Gamma-final}
 \Gamma_K(D)
 \le C(D+1)^h(D+1)^{s+1}
 =C(D+1)^{r+1}.
\end{equation}
Take $e=\ceil{D/2}$. Then
$\mu(D,e)=\ceil{D/2}$,
and, uniformly for $q\ge2$,
$\varepsilon_K(D,e,q) \le C(D+1)^{r+2}2^{-\ceil{D/2}}$.
The last expression is at most $1/2$ for every $D\ge D_0$, where $D_0$ depends only on the controlled chart and on $r$. Equations \eqref{eq:moving-recursion-solved} and \eqref{eq:moving-Gamma-final} prove \eqref{eq:moving-local}. All constants used above arise from the uniform integral chart data, the fixed integer $r$, and degree inequalities; hence the same $C,D_0$ work for every allowed finite-field fiber.
\end{proof}

\begin{proof}[Proof of \cref{thm:ci-global}]
Choose $p_0$ larger than every prime divisor of $N_0$, so that the controlled charts of \cref{prop:controlled-charts} exist on every $\cX_p$ with $p\ge p_0$. By the uniformity assertion in \cref{cor:moving-local}, the finitely many controlled integral charts and the finitely many values $0\le h<c$ admit a common moving-chart threshold; choose $d_0$ at least that threshold. On each standard chart, dehomogenization carries the uniform distribution on homogeneous forms to the uniform distribution on the corresponding polynomial space.

Relabel the equations so that
$d_1\le\cdots\le d_c$.
This changes neither their product distribution up to relabelling nor the final degeneracy locus. Let $\Sigma_i$ be the rank-degeneracy locus for the first $i$ equations, and put
$E_i=\set{\dim\Sigma_i\ge1}$ and $E_0=\varnothing$.

Fix $i$ and condition on $f_1,\ldots,f_{i-1}$. Suppose $E_{i-1}$ does not occur. Base change all loci to $\overline{\F}_p$, in accordance with the convention for dimension. Let $Z$ be a positive-dimensional irreducible component of $(\Sigma_{i,\overline{\F}_p})_{\mathrm{red}}$. At the generic point $\eta_Z$, the first $i-1$ columns of the base-changed map \eqref{eq:global-ci-differential-map} have rank $i-1$; otherwise $\eta_Z$ would belong to $\Sigma_{i-1,\overline{\F}_p}$, whose dimension is at most zero. Choose a controlled standard chart whose base change contains $\eta_Z$; trivialize the relevant line bundles there, and write $G_a$ for the dehomogenization of $f_a$. One of the finitely many $(i-1)\times(i-1)$ Jacobian minors is nonzero at $\eta_Z$; choose the corresponding pivot $K$. The pivot function and the associated open $U_K$ are defined over $\F_p$.

On the associated $U_K$, the tangent bundle is the common kernel of
$dG_1,\ldots,dG_{i-1}$.
Consequently, on $U_K\cap V(G_i)$,
$\operatorname{rank}(dG_1,\ldots,dG_i)<i$
if and only if $dG_i$ vanishes on $TU_K$. By \eqref{eq:pivot-duality}, this is equivalent to
$E_{K,j}G_i=0\qquad(j\in J)$.
The intersection $Z\cap (U_K)_{\overline{\F}_p}$ is a nonempty open subset of the irreducible scheme $Z$; it is therefore dense and has dimension $\dim Z\ge1$. Thus $\BSing_{U_K}(G_i)$ is positive-dimensional.

Because $d_a\le d_i$ for $a<i$, \cref{cor:moving-local} applies with $D=d_i$. At the $i$th stage there are $M$ controlled charts and, on each chart, at most $\binom{r}{i-1}$ possible pivots. A union bound over these at most $M\binom{r}{i-1}$ events, with this fixed factor absorbed into the constant, gives
\begin{equation}\label{eq:ci-induction-step}
 \Prob(E_i\mid f_1,\ldots,f_{i-1})
 \le\one_{E_{i-1}}
 +C(d_i+1)^{r+1}p^{-\ceil{d_i/2}}.
\end{equation}
Taking expectations and iterating proves \eqref{eq:ci-sum}.

It remains to derive \eqref{eq:ci-min}. Put $a=r+1$. There are constants $D_a,C_a$ such that, for every $p\ge2$ and integers $d\ge D\ge D_a$,
\begin{equation}\label{eq:envelope}
 (d+1)^ap^{-\ceil{d/2}}
 \le C_a(D+1)^ap^{-\ceil{D/2}}.
\end{equation}
Indeed, group the integers into the pairs $2k-1,2k$. The maximum on each pair is
$M_k=(2k+1)^ap^{-k}$.
For all sufficiently large $k$,
\[
 \frac{M_{k+1}}{M_k}
 =\left(\frac{2k+3}{2k+1}\right)^a\frac1p
 \le1
\]
uniformly for $p\ge2$. Thus all terms with $d\ge D$ are bounded by the maximum on the pair containing $D$. If $D=2k$, this maximum equals the right side of \eqref{eq:envelope} with $C_a=1$; if $D=2k-1$, it is at most $(3/2)^a$ times that right side. After enlarging $d_0$ once more so that $d_0\ge D_a$, apply \eqref{eq:envelope} with $D=d_{\min}$ to every summand in \eqref{eq:ci-sum} and absorb the fixed number $c$ into $C'$.
\end{proof}

\subsection{Consequences, comparisons, and scope}\label{sec:consequences}

\subsubsection{Fixed powers and exponential decay}

\begin{corollary}[Fixed-power consequences]\label{cor:fixed-powers}
For every $A>0$, after increasing the relevant degree threshold, the probabilities in \eqref{eq:jet-global} and \eqref{eq:ci-min} are at most $p^{-A}$, uniformly for $p\ge p_0$. In the complete-intersection case, the estimate is uniform over all degree vectors satisfying the prescribed lower bound on $d_{\min}$.
\end{corollary}

\begin{proof}
The exponents $\lambda_m(d)$ and $\ceil{d_{\min}/2}$ grow linearly in the relevant minimum degree, whereas the remaining factors are polynomial. Since $p\ge2$, a fixed positive proportion of either linear exponent absorbs the corresponding polynomial factor uniformly in $p$.
\end{proof}

More precisely, for fixed $m$ and $0<\eta<m/(m+1)$, there is a degree threshold such that
\begin{equation}\label{eq:jet-exponential}
 \Prob\bigl(\dim J_m(f)\ge1\bigr)\le p^{-\eta d}.
\end{equation}
For complete intersections, every $0<\eta<1/2$ similarly gives
\begin{equation}\label{eq:ci-exponential}
 \Prob\bigl(\dim\Sigma_c(\mathbf f)\ge1\bigr)
 \le p^{-\eta d_{\min}}.
\end{equation}
In particular, the first-order choice $\eta=1/3$ gives a bound of the form $p^{-d/3}$.

\subsubsection{The filtered quotient}

After the decomposition $F=F_0+Q(t_j)H$, the term $F_0$ is uniform in the filtered remainder space, not in the full polynomial space. Replacing it by a new full-space random polynomial would change the recursive distribution. By \cref{lem:filtered-normal-form}, the spaces $\cR_I(d)$ identify with the images of the degree filtration in $A/(Q_i(t_i):i\in I)$, and
\begin{equation}\label{eq:filtered-summary}
 \cR_I(d)=
 \bigoplus_{a\ge0}Q(t_j)^a\cR_{I\cup\set{j}}(d-ae)
\end{equation}
shows that the random blocks retain their exact product distribution at every step.

\subsubsection{The two roles of an orbit}

A degree-$e$ orbit is used twice. Geometrically, the boundary-value estimate and disjointness of root orbits show that the probability of missing a controlled curve is $O(\Gamma_m(d)q^{-e})$. Algebraically, every closed point of the slice has $\F_{q^e}\hookrightarrow\kappa(y)$, so a $Q$-adic block supplies up to $e$ independent evaluation conditions. The first $m$ normal blocks contribute $\Lambda_m(d,e)=\min\{me,d-e+1\}$. Balancing the two linear terms gives the proportion $m/(m+1)$; for $m=1$, it gives the half-degree exponent.

\subsubsection{Comparison with the $p$-power method}

Poonen's derivative-decoupling method and its variants use
$\frac{\partial}{\partial t_i}(g^pt_i)=g^p$.
The available auxiliary degree is $\floor{(d-1)/p}$. This is effective when $d/p\to\infty$, but it remains bounded when $d\asymp p$; see \cite{Poonen2004,Poonen2007,BucurKedlaya2012,AsgarliLoveYip2025}. Orbit slicing instead uses
$\partial_i\bigl(Q(t_i)H\bigr)|_{Q(t_i)=0}=Q'(t_i)H$,
where $e=\deg Q$ is chosen independently of the characteristic. The full $Q$-adic expansion gives the corresponding triangular identities for higher normal Taylor coefficients.

\subsubsection{The projective and quasiprojective cases}

In the projective case, a fixed ample divisor meets every positive-dimensional closed subvariety, and one may organize the induction using deterministic divisors. This is the relevant geometric input in Wang's $p^{-2}$ estimate under the fiber hypotheses of \cite[Lemma~6.3]{Wang2022}. An affine curve can avoid a fixed divisor through the boundary. Here the deterministic intersection statement is replaced by the estimate that a random orbit slice misses a bounded-degree affine curve with probability $O(\Gamma_m(d)q^{-e})$. The boundary-value lemma makes this estimate uniform.

\subsubsection{A finite-field local form}

The local proof has a formulation independent of arithmetic spreading out. Fix $m\ge1$ and assume $\operatorname{char}\F_q>m$. Let $U\subseteq\A^n_{\F_q}$ be smooth, with fixed ambient coordinates $z_1,\ldots,z_n$, and suppose that: distinct ambient coordinates $t_i=z_i$ give an \'{e}tale map $U\to\A^r_{\F_q}$; derivations $D_i$ satisfy $D_i(t_j)=\Delta_i\delta_{ij}$ with each $\Delta_i$ a unit; the $D_i$ have polynomial representatives of uniformly bounded degree in the same ambient variables; and there is a projective closure with $U=\overline U\cap D_+(B)$, $\deg B\le b$, and $\gdeg(\overline U_{\overline{\F}_q})\le\delta$. All filtrations and all equations $Q(t_i)$ are measured in the original ambient coordinates. This last condition matters: for a general polynomial \'{e}tale coordinate, one may have $\deg Q(t_i)\ne\deg Q$, and the filtered direct sum \eqref{eq:filtered-direct-sum} need not hold with the stated degree filtration. Under these hypotheses, \cref{cor:balanced-local-jet} is a finite-field theorem. The arithmetic result follows by constructing finitely many such charts with uniform data, and \cref{cor:moving-local} allows these data to vary polynomially with previously chosen equations.

\subsubsection{Dependencies and scope}

The proof does not use the Weil conjectures, Lang--Weil estimates, or the $abc$ conjecture. Its inputs are spreading out and \'{e}tale criteria, projective degree inequalities, normalization of curves and boundary divisors, the elementary count of irreducible polynomials over finite fields, and linear algebra in filtered quotient spaces. The results concern positive-dimensional jet-zero and rank-degeneracy loci. They do not supply the separate closed-point analysis needed to compute the full density of everywhere regular arithmetic hypersurface sections in \cite{Poonen2004}; the first-order theorem supplies the uniform estimate isolated in Conjecture~5.2.

\subsubsection{Relation to previous work}\label{sec:literature}

Poonen's finite-field Bertini theorem separates singular points according to residue degree and uses characteristic-$p$ derivative decoupling \cite{Poonen2004,Poonen2007}. Complete-intersection and semiample variants retain this sieve structure \cite{BucurKedlaya2012,ErmanWood2015}. Results with prescribed subschemes or local conditions alter the linear system while continuing to use closed-point sieves \cite{Poonen2008,Wutz2016,Gunther2017}. Motivic Euler products and Taylor-condition theorems allow broader local conditions \cite{BiluHowe2021,Bertucci2026}; their principal asymptotic parameter is the degree in a fixed finite-field geometry.

Poonen--Slavov use random hyperplane sections to control exceptional loci in Bertini irreducibility, and Kmentt--Shute treat higher-codimensional linear sections \cite{PoonenSlavov2022,KmenttShute2022}. Those arguments use linear sections and point-counting statistics. The slice here is the nonlinear closed-point fiber $Q(t_i)=0$. The same defining equation provides the normal Taylor block, and iteration uses the exact filtered decomposition \eqref{eq:q-adic-sum}.

Slavov and Tseng study parameter spaces of hypersurfaces with positive-dimensional singular loci and analyze large components of the corresponding discriminant strata \cite{Slavov2015,Tseng2020}. Lindner relates defect to large singularity and obtains finite-field density consequences \cite{Lindner2020}. Poonen--Stoll prove lower bounds for discriminant valuations in degenerations with multiple or positive-dimensional singularities \cite{PoonenStoll2025}. These results address different parameter regimes or different local questions.

Asgarli--Love--Yip prove that suitable Galois orbits impose the expected number of independent conditions on hypersurfaces \cite{AsgarliLoveYip2025}. Here the orbit is sampled as a coordinate slice of a singular component, and its defining equation is also used in the polynomial expansion. Zhang--Yang's anti-Bertini examples concern degree-one sections after varying the embedding \cite{ZhangYang2026}; they do not address high-degree density for a fixed embedding. Ajit--Bertucci study preservation of Hilbert--Samuel multiplicity under random sections \cite{AjitBertucci2026}.

The distinguishing ingredients of the present argument are the irreducible coordinate slice $Q(t_i)=0$, the boundary-controlled curve-hitting estimate, the residue-field inclusion $\F_{q^e}\subseteq\kappa(y)$, the normal $Q$-adic Taylor tower, and the filtered recursion preserving the exact distribution. Jacobian-pivot charts reduce the complete-intersection statement to the first-order case. This comparison is intended to identify the mechanism, not to make an exhaustive priority claim.

\subsubsection*{Conclusion}

Frobenius-orbit slicing supplies an auxiliary degree independent of the characteristic. The geometric estimate controls the probability that the slice misses a curve in a positive-dimensional jet-zero locus, while the same slice yields independent normal Taylor conditions. Their balance gives the exponent $\floor{m(d+1)/(m+1)}$. The first-order specialization proves the arithmetic Bertini estimate in \cite[Conjecture~5.2]{Poonen2004}. For complete intersections, Jacobian-pivot charts reduce each new positive-dimensional rank defect to the first-order hypersurface estimate and give the termwise bound \eqref{eq:ci-sum}, hence the uniform envelope \eqref{eq:ci-min}.

\section*{Funding}
No funding was received for this work.

\section*{Declaration of competing interest}
The authors declare no competing interests.

\section*{Declaration of Generative AI and AI-Assisted Technologies in the Writing Process}
The author used GPT 5.5 Pro to enumerate possible cases, search for counterexamples, and generate preliminary proofs. All AI-generated suggestions were manually reviewed, modified, and independently verified by the author. No unverified AI-generated proof or mathematical claim was incorporated into the final manuscript.

\printcredits

\begingroup
\sloppy

\endgroup

\end{document}